\pdfoutput=1 
\documentclass[]{article}
\usepackage[english]{babel}
\usepackage{tikz-cd}
\usepackage[utf8x]{inputenc}
\usepackage[T1]{fontenc}

\usepackage[a4paper,top=3cm,bottom=2cm,left=3cm,right=3cm,marginparwidth=1.75cm]{geometry}

\usepackage{enumitem} 
\usepackage{scalerel,stackengine}
\usepackage{tikz-cd}
\usepackage{bm}
\usepackage{url}
\usepackage{amsmath}
\usepackage{amssymb}
\usepackage{amsthm}
\usepackage{mathtools}
\usepackage{amsmath,calligra,mathrsfs}
\usepackage{yhmath}
\usepackage{listofitems}

\usepackage{graphicx}
\usepackage[colorinlistoftodos]{todonotes}
\usepackage[colorlinks=true, allcolors=blue]{hyperref}
\newcommand{\uset}[1]{\underset{#1}{\otimes}{}}
\DeclareMathOperator{\Deriv}{\mathscr{D}\text{\kern -3pt {\calligra\large er}}\,}
\DeclareMathOperator{\Endo}{\mathscr{E}\text{\kern -3pt {\calligra\large nd}}\,}
\DeclareMathOperator{\Hom}{Hom \text{\kern -3pt }\,}
\DeclareMathOperator{\Der}{Der \text{\kern -3pt }\,}
\DeclareMathOperator{\End}{End \text{\kern -3pt }\,}
\DeclareMathOperator{\Coh}{Coh \text{\kern -2pt }\,}
\DeclareMathOperator{\QCoh}{QCoh \text{\kern -2pt }\,}
\DeclareMathOperator{\res}{res \text{\kern -1pt }\,}
\DeclareMathOperator{\act}{act \text{\kern -1pt }\,}
\DeclareMathOperator{\Loc}{Loc \text{\kern -2pt }\,}
\DeclareMathOperator{\Lie}{Lie \text{\kern -2pt }\,}
\DeclareMathOperator{\loc}{loc \text{\kern -2pt }\,}
\DeclareMathOperator{\ob}{ob \text{\kern -2pt }\,}
\DeclareMathOperator{\ad}{ad \text{\kern -2pt }\,}
\DeclareMathOperator{\fg}{fg \text{\kern -2pt }\,}
\DeclareMathOperator{\co}{co \text{\kern -3pt }\,}
\DeclareMathOperator{\Mod}{Mod \text{\kern -2pt }\,}
\DeclareMathOperator{\Ann}{Ann \text{\kern -2pt }\,}
\DeclareMathOperator{\Sym}{Sym \text{\kern -2pt }\,}
\DeclareMathOperator{\Dif}{Dif \text{\kern -2pt }\,}
\DeclareMathOperator{\Ad}{Ad \text{\kern -2pt }\,}
\DeclareMathOperator{\gr}{gr \text{\kern -2pt }\,}
\DeclareMathOperator{\im}{im \text{\kern -2pt }\,}
\DeclareMathOperator{\coim}{coim \text{\kern -2pt }\,}
\DeclareMathOperator{\id}{id \text{\kern -2pt }\,}
\DeclareMathOperator{\bimod}{bimod \text{\kern -2pt }\,}
\DeclareMathOperator{\height}{ht \text{\kern -2pt }\,}
\DeclareMathOperator{\tensor}{tensor \text{\kern -2pt }\,}
\DeclareMathOperator{\coker}{coker \text{\kern -2pt }\,}
\newcommand{\tig}{\operatorname{\tilde{\textit{i}_{\mathfrak{g}}}}}
\DeclareMathOperator{\Spec}{Spec \text{\kern -2pt }\,}

\newtheorem{theorem}{Theorem}[section]
\newtheorem{lemma}[theorem]{Lemma}
\newtheorem{proposition}[theorem]{Proposition}
\newtheorem{corollary}[theorem]{Corollary}
\newtheorem{question}[theorem]{Question}
\newtheorem{definition}[theorem]{Definition}
\newtheorem{observation}[theorem]{Observation}

\newtheorem{assumption}[theorem]{Assumption}
\title{Towards Affinoid Duflo's Theorem I: Twisted differential operators}
\author{Ioan Stanciu}
\date{}
\begin{document}
\maketitle

\section{Introduction}

This paper is the first of a series of papers that aims to answer Question A from \cite{Munster} regarding the classification of primitive ideals in the affinoid enveloping algebra of a semisimple Lie algebra defined over a discrete valuation ring.


\begin{theorem}
\label{d1Dutheorem}
Let $\mathfrak{g}$ be a semisimple Lie algebra defined over a field $K$ of characteristic $0$. Then any primitive ideal in $U(\mathfrak{g})$ with $K$-rational infinitesimal central character is the annihilator of the simple quotient of some Verma module. In case $K=\mathbb{C}$, the theorem gives a classification of all primitive ideals.
\end{theorem} 

Assume for now that our ground field is $\mathbb{C}$ and let $\mathfrak{g}$ be a semisimple $\mathbb{C}$-Lie algebra. For a complex variety $X$, we will denote $\mathcal{D}_X$ the sheaf of differential operators on $X$. Let $G$ be the semisimple connected affine algebraic group associated with $\mathfrak{g}$ and fix $B$ a Borel subgroup. Let $X=G/B$ be the flag variety associated to $\mathfrak{g}$. In \cite[Proposition 3.6]{BoBr}, the authors prove an equivalence of categories between $G$-equivariant coherent $\mathcal{D}_{X \times X}$-modules and $B$-equivariant coherent $\mathcal{D}_{X}$-modules. Further, the Beilinson-Bernstein theorem \cite{BB} establishes an equivalence of categories between coherent $\mathcal{D}_X$-modules and finitely generated $U(\mathfrak{g})$-modules with trivial central character. Combining these two results, one obtains a geometric proof of Duflo's theorem for ideals with \emph{trivial central character}. In the next article, we remove some restrictions from \cite[Proposition 3.6]{BoBr}: we prove that the results hold over a general commutative Noetherian ring and that the equivalence holds between coherent modules over certain \emph{homogeneous sheaves of twisted differential operators}, which can be regarded as  \emph{equivariant} twisted differential operators.

There is a well established theory of twisted differential operators and homogeneous twisted differential operators over a \emph{complex} variety introduced in \cite{BB2} and treated in more detail in \cite{Mil2}. The authors also explore the connections between twisted differential operators(tdo's) and Picard algebroids and the following question is answered:

\begin{question}
Let $f:Y \to X$ be a map of smooth complex varieties and let $\mathcal{D}$ be a tdo on $X$. How should we define the pullback of $\mathcal{D}$, call it $f^{\bullet} \mathcal{D}$, such that $f^{\bullet} \mathcal{D}$ is a tdo on $Y$?

\end{question}

The solution proposed in \cite{BB2} was to define $f^{\bullet} \mathcal{D}$ using the $\Dif$ functor introduced by Grothendieck: 
                        $$f^{\bullet} \mathcal{D}:=\Dif_{f^{-1} \mathcal{D}}(f^* \mathcal{D},f^*\mathcal{D})$$

is the sheaf of differential operators from $f^{*} \mathcal{D}$ to itself that commute with the right $f^{-1} \mathcal{D}$-action. In particular when $\mathcal{D}=\mathcal{D}_X$, we obtain $f^{\bullet} \mathcal{D}_X= \mathcal{D}_Y$.

Now, let $R$ be a commutative base ring and $X$ be an $R$-scheme that is smooth, separated and locally of finite type. In order to build a good theory of twisted differential operators over a commutative ring, there are two basic questions we need to answer: 

\textbf{Question}
What constitutes a good definition of a tdo on $X$?

\textbf{Question}
Given $f:Y \to X$ a map of smooth, separated and locally of finite type $R$-schemes and $\mathcal{D}$ a tdo on $X$, how should we define the pullback of $\mathcal{D}$ such that it is also a tdo on $Y$?

There are two possible candidates of sheaves of differential operators that one can define over $X$: $\mathcal{D}_X$-the sheaf of crystalline differential operators and $\mathscr{D}_X$-the sheaf of Grothendieck's differential operators.

One of the key properties satisfied by twisted differential operators  over complex varieties is that they come equipped with a filtration such that the associated graded is isomorphic with the symmetric algebra of the tangent sheaf over the ring of functions. The sheaf of Grothendieck's differential operators $\mathscr{D}_X$ has a natural filtration given by the order of differential operators, but the associated grading ring does not satisfy the desired property. Therefore, we choose to work with the sheaf of crystalline differential operators.


Attempting working with the classical definition of the pullback of tdo's we immediately encounter a problem. Assume that $Y=\mathbb{A}^1$ is the affine line over $R$, let $X=\Spec R$ be the base of $Y$ and $f:Y \to X$ be the natural projection. Let $\mathcal{D}_X \cong R$ be the sheaf of crystalline differential operators on $X$. Then using the classical definition we obtain

             $$f^{\bullet} \mathcal{D} \cong \mathscr{D}_{\mathbb{A}^1}, $$

which is the sheaf of \emph{Grothendieck's differential operators}. In particular, if we work with the classical definition we obtain that the pullback of a twisted differential operator does not satisfy the desired property.

To resolve the problem with the definition, we will explore the correspondence between twisted differential operators (shortened tdo's frow now on) and \emph{Picard algebroids} and define the pullback of tdo's by first defining the pullback of Picard algebroids. 

Specifically, we will be able to go from tdo's to Picard algebroids and vice-versa using two functors $\Lie$ and $\mathscr{T}$. For a Picard algebroid $\mathcal{L}$ on $X$, we can consider the pullback $f^{\#} \mathcal{L}$ as a Lie algebroid on $Y$ and for a general tdo on $X$ we define

                $$f^{\#} \mathcal{D}:=\mathscr{T}(f^{\#}(\Lie(\mathcal{D}))).$$
In the case when $R=\mathbb{C}$, our definition coincides with the definition found in \cite{BB2} and \cite{Mil2}.

\textbf{Statement of the main results of the article} 

Let $G$ be a smooth affine algebraic group of finite type over $\Spec R$ and let $f:Y \to X$ be a locally trivial $G$-torsor. Then for a tdo $\mathcal{D}$ equipped with a suitable $G$-action (we call this a $G$-htdo), we may define its descent $f_{\#}\mathcal{D}^G$, which is a tdo on $X$.  

\begin{proposition}[Corollary \ref{d1descentGhtdocorresp}]
\label{d1introprop}
The functors $f_{\#}(-)^G$ and $f^{\#}(-)$ induce quasi-inverse equivalences between $G$-htdo's on $Y$ and tdo's on $X$.

\end{proposition}

For a $G$-htdo $\mathcal{D}$ we may define the $\Coh(\mathcal{D},G)$ the category of $G$-equivariant coherent $\mathcal{D}$-modules. Further, let $\mathcal{A}$ be a tdo on $X$ and $\mathcal{M}$ a $\mathcal{A}$-module. Then we may endow the $\mathcal{O}$-module pullback $f^*\mathcal{M}$ with the structure of a $f^{\#}\mathcal{A}$-module and we call this module $f^{\#} \mathcal{M}$.

\begin{theorem}[Theorem \ref{d1maintheorem}]
\label{d1introtheorem}
Assume that $R$ is a Noetherian ring. Let $G$ be a smooth affine algebraic group of finite type. Let $X,Y$ be smooth separated and locally of finite type $R$-schemes and let $f:Y \to X$ be a locally trivial $G$-torsor. Further, let $\mathcal{D}$ be a sheaf of $G$-homogeneous twisted differential operators on $Y$. The functors:

\begin{equation}
\begin{split}
&f_* (-)^G: \Coh(\mathcal{D},G) \to \Coh(f_{\#} \mathcal{D}^G)\\
&f^{\#}(-): \Coh(f_{\#} \mathcal{D}^G) \to \Coh(\mathcal{D},G).
\end{split} 
\end{equation}

are quasi-inverse equivalences of categories between coherent $G$-equivariant $\mathcal{D}$-modules and coherent $(f_{\#} \mathcal{D})^G$-modules.
\end{theorem}

In fact, we generalise the proposition and the theorem in two directions. First, for $r \in R$ a regular element, we define the notion of $r$-deformed Picard algebroids and $r$-deformed $G$-htdo's. Then we may prove that the descent of an $r$-deformed $G$-htdo is an $r$-deformed tdo and similarly the theorem holds for coherent modules over $r$-deformed $G$-htdo's.

Secondly, for another smooth affine algebraic group of finite type $B$ acting on $X$ and $Y$, we prove that for a $G \times B$-htdo on $Y$, its descent under a $B$-equivariant $G$-torsor is a $B$-htdo. A similar result also holds for the corresponding coherent $G \times B$-modules. This generality will be needed in future applications: in \cite{Sta2} we use this framework to give a geometric proof of \ref{d1Dutheorem} and in \cite{Sta3} we prove an affinoid version of the same theorem.

\textbf{Structure of the paper}

In Sections \ref{d1sectionequivariantOmodules} and \ref{d1sectionequivariantdescentOmod}, we review the theory of equivariant $\mathcal{O}$-modules and equivariant descent for a locally trivial torsor. Next, we define in Section \ref{d1sectiontdo's} the sheaf of crystalline differential operators and the notion of $r$-deformed twisted differential operators on a smooth, separated and locally of finite type $R$-scheme. In the next two sections, we establish correspondences between $r$-deformed Picard algebroids/equivariant Picard algebroids and $r$-deformed twisted differential operators/homogeneous twisted differential operators. We then define the pullback of $r$-deformed Picard algebroids and $r$-deformed tdo's in Section \ref{d1sectionpullbackofLiealg}. In the next two sections, we explore the connections between modules over $r$-deformed Picard algebroids and $r$-deformed twisted differential operators.

Finally, in Section \ref{d1sectiondescentliealg}, we prove equivariant descent for $r$-deformed homogeneous twisted differential operators (Proposition \ref{d1introprop}) and in Section \ref{d1sectiontdomoddescent}, we prove equivariant descent for modules over $r$-deformed homogeneous twisted differential operators (Theorem \ref{d1introtheorem}).

\textbf{Conventions}

Throughout this document, $R$ will denote a commutative ring of arbitrary characteristic and all the schemes will be $R$-schemes. For a map $f:Y \to X$ of $R$-schemes, we will denote $f^*$ the pullback in the category of $\mathcal{O}$-modules and $f_*$ the pushforward sheaf. Unadorned tensor products will be assumed to be taken over $R$. An element $r \in R$ is called regular if it is not a zero divisor.

\section{Equivariant \texorpdfstring{$\mathcal{O}$}{O}-modules}
\label{d1sectionequivariantOmodules}

Let $G$ be an affine algebraic group scheme acting on a  scheme $X$; denote the action by $\sigma_X: G \times X \to X$. Furthermore, we denote $p_X:G \times X \to X$ and $p_{2X}:G \times G \times X \to X$ the projections onto the $X$ factor, $p_{23X}:G \times G \times X \to G \times X$ the projection onto the second and third factor and $m:G \times G \to G$ the multiplication of the group $G$.

\begin{definition}

Let $G$ an algebraic group scheme acting on a scheme $X$. A $G$-equivariant $\mathcal{O}_X$-module is a pair $(\mathcal{M},\alpha)$, where $\mathcal{M}$ is a quasi-coherent $\mathcal{O}_X$-module and $\alpha:\sigma_X^*\mathcal{M} \to p_X^*\mathcal{M}$ is an isomorphism of $\mathcal{O}_{G \times X}$-modules such that the diagram

\begin{center}
\begin{tikzcd}

&(1_G \times \sigma_X)^*p_X^*\mathcal{M} \arrow[r," p_{23X}^* \alpha"]   &p_{2X}^*\mathcal{M} \\
&(1_G \times \sigma_X)^* \sigma_X^* \mathcal{M} \arrow [u,"(1_G \times \sigma_X)^* \alpha "] \arrow[r,leftrightarrow,"id "] &(m \times 1_X)^*\sigma_X^* \mathcal{M} \arrow[u, " (m \times 1_X)^* \alpha"]
\end{tikzcd}
\end{center}

of $\mathcal{O}_{G \times G \times X}$-modules commutes (the cocycle condition) and the pullback 
                        $$(e \times 1_X)^* \alpha: \mathcal{M} \to \mathcal{M}$$
is the identity map.
\end{definition}

We prove a crucial lemma that will be used in the future; it is stated on the stack project, but the proof is omitted.
\begin{lemma}\cite[\href{https://stacks.math.columbia.edu/tag/03LG}{03LG}]{StackProject}
\label{d1Oequivpreservefunctor}

Let $G$ be an affine algebraic group acting on schemes $X$ and $Y$ and let $f:Y \to X$ be a $G$-equivariant morphism. Then the pullback functor $f^*$ given by $$(\mathcal{M},\alpha) \mapsto (f^*\mathcal{M},(1_G \times f)^* \alpha)$$ defines a functor from $G$-equivariant $\mathcal{O}_X$-modules to $G$-equivariant $\mathcal{O}_Y$-modules.

\end{lemma}

\begin{proof}
Let $\mathcal{M}$ be a $G$-equivariant $\mathcal{O}_X$-module. Since $\alpha:\sigma_X^*\mathcal{M} \to p_X^* \mathcal{M}$ is an isomorphism, we get that $(1_G \times f)^* \alpha: (1_G \times f)^* \sigma_X^* \mathcal{M} \to (1_G \times f)^* p_X^* \mathcal{M}$ is also an isomorphism. We have

\begin{equation}
\begin{split}
&\sigma_X \circ (1_G \times f)(g,y)=gf(y)=f(gy)=f \circ \sigma_Y(g,y), \text{ so } (1_G \times f)^* \sigma_X^* \mathcal{M}=\sigma_Y^* f^* \mathcal{M}.   \\
&p_X \circ (1_G \times f)(g,y)=f(y)= f \circ p_Y, \text{ so } (1_G \times f)^* p_X^* \mathcal{M}= p_Y^* f^* \mathcal{M}.
\end{split}
\end{equation}

Thus, $(1_G \times f)^* \alpha: \sigma_Y^*(f^* \mathcal{M}) \to  p_Y^*(f^* \mathcal{M}) \text{ is an isomorphism}.$ 

Next, we need to prove that the morphism $(1_G \times f)^* \alpha$ satisfies the cocycle condition; that is we need to show that the diagram

\begin{equation}
\label{d1GequivariantYdiagram}
\begin{tikzcd}
&(1_G \times \sigma_Y)^*p_Y^* f^*\mathcal{M} \arrow[r," p_{23Y}^* (1_G \times f)^* \alpha"]   &p_{2Y}^*f^*\mathcal{M} \\
&(1_G \times \sigma_Y)^* \sigma_Y^*f^* \mathcal{M} \arrow [u,"(1_G \times \sigma_Y)^* (1_G \times f)^* \alpha "] \arrow[r,leftrightarrow,"\id "] &(m \times 1_Y)^*\sigma_Y^* f^* \mathcal{M} \arrow[u, " (m \times 1_Y)^* (1_G \times f)^* \alpha"]
\end{tikzcd}
\end{equation}

of $\mathcal{O}_{G \times G \times Y}$-modules commutes given that the diagram

\begin{equation}
\label{d1GequivariantXdiagram}
\begin{tikzcd}
&(1_G \times \sigma_X)^*p_X^*\mathcal{M} \arrow[r," p_{23X}^* \alpha"]   &p_{2X}^*\mathcal{M} \\
&(1_G \times \sigma_X)^* \sigma_X^* \mathcal{M} \arrow [u,"(1_G \times \sigma_X)^* \alpha "] \arrow[r,leftrightarrow,"\id "] &(m \times 1_X)^*\sigma_X^* \mathcal{M} \arrow[u, " (m \times 1_X)^* \alpha"]
\end{tikzcd}
\end{equation}

of $\mathcal{O}_{G \times G \times X}$-modules commutes.

We shall prove that the diagram \ref{d1GequivariantYdiagram} is the pullback of the diagram \ref{d1GequivariantXdiagram} under the morphism $(1_G \times 1_G \times f)$. We have that

\begin{equation}
\begin{split}
&f \circ  p_Y \circ  (1_G \times \sigma_Y)(g_1,g_2,y)=f(g_2y)=p_X \circ (1_G \times \sigma_X) \circ (1_G \times 1_G \times f) (g_1,g_2,y). \\
&f \circ p_{2Y}(g_1,g_2,y)=y=p_{2X} \circ (1_G \times 1_G \times f)(g_1,g_2,y). \\
&f \circ \sigma_Y \circ (1_G \times \sigma_Y) (g_1,g_2,y)= f(g_1g_2y)=\sigma_X \circ (1_G \times \sigma_X) \circ (1_G \times 1_G \times f) (g_1,g_2,y). \\
&f \circ \sigma_Y \circ ( m \times 1_Y)(g_1,g_2,y)=g_1g_2f(y)= \sigma_X \circ (m \times 1_X) \circ (1_G \times 1_G \times f) (g_1,g_2,y).
\end{split}
\end{equation}

Therefore,

\begin{equation}
\label{d11G1Gfobjects}
\begin{split}
((1_G \times \sigma_Y)^*p_Y^*) f^* \mathcal{M} &= (1_G \times 1_G \times f)^* ((1_G \times \sigma_X)^*p_X^* \mathcal{M}). \\
p_{2Y}^*(f^* \mathcal{M}) &= (1_G \times 1_G \times f)^* (p_{2X}^* \mathcal{M}). \\
((1_G \times \sigma_Y)^* \sigma^*) f^* \mathcal{M} &= (1_G \times 1_G \times f)^* ((1_G \times \sigma_X)^*\sigma_X^* \mathcal{M}).\\
((m \times 1_Y)^* \sigma^*) f^* \mathcal{M} &= (1_G \times 1_G \times f)^* ((m \times 1_X)^* \sigma_X^* \mathcal{M}).
\end{split}
\end{equation}

Similarly, we get 

\begin{equation}
\begin{split}
&(1_G \times f)\circ p_{23Y}(g_1,g_2,y)=(g_2,f(y))=p_{23X} \circ (1_G \times 1_G \times f) (g_1,g_2,y). \\
&(1_G \times f) \circ (1_G \times \sigma_Y) (g_1,g_2,y)=(g_1,f(g_2y))= (1_G \times \sigma_X) \circ (1_G \times 1_G \times f) (g_1,g_2,y). \\
&(1_G \times f) \circ (m \times 1_Y)(g_1,g_2,y)=(g_1g_2,f(y)= (m \times 1_X) \circ (1_G \times 1_G \times f) (g_1,g_2,y).
\end{split}
\end{equation}

Thus, 

\begin{equation}
\label{d11G1Gfmorphism}
\begin{split}
p_{23Y}^* (1_G \times f)^* \alpha &= (1_G \times 1_G \times f)^* p_{23X}^* \alpha. \\
(1_G \times \sigma)^* (1_G \times f)^* \alpha &= (1_G \times 1_G \times f)^* (1_G \times \sigma_X)^* \alpha. \\
(m \times 1_X)^*(1_G \times f)^* \alpha &= (1_G \times 1_G \times f)^* (m \times 1_X)^* \alpha.
\end{split}
\end{equation}

Combining equations \eqref{d11G1Gfobjects} and \eqref{d11G1Gfmorphism}, we get that the diagram \eqref{d1GequivariantYdiagram} is indeed the pullback of the diagram \eqref{d1GequivariantXdiagram} under the morphism $1_G \times 1_G \times f$, so $(1_G \times f)^* \alpha$ satisfies the cocycle condition. 

Finally, we need to prove  that the map $(e \times 1_Y)^*(1_G \times f)^*\alpha:f^*\mathcal{M} \to f^*\mathcal{M}$  is the identity map using   $(e \times 1_X)^* \alpha: \mathcal{M} \to \mathcal{M}$ is the identity map. 

We have that $(1_G \times f)\circ (e \times 1_Y) (g,y)=(e,f(y))=(e \times 1_X) (1_G \times f)$, so

$$ (e \times 1_Y)^* (1_G \times f)^* \alpha = (1_G \times f)^* (e \times 1_X)^* \alpha = (1_G \times f)^* (\id)= \id,$$

since the identity map is preserved by any functor.
\end{proof}

\begin{definition}
Let $G$ an affine algebraic group acting on a scheme $X$ via $\sigma_X$. We define the category of $G$-equivariant quasi-coherent $\mathcal{O}_X$-modules. Objects are given by $G$-equivariant $\mathcal{O}_X$-modules.

A morphism of $G$-equivariant $\mathcal{O}_X$-modules $(\mathcal{M},\alpha_M)$ and $(\mathcal{N},\alpha_N)$ is a map $\phi \in \Hom_{\mathcal{O}_X}(\mathcal{M},\mathcal{N})$ such that the following diagram commutes:
\begin{center}
\begin{tikzcd}
&\sigma_X^* \mathcal{M} \arrow[d,"\sigma_X^*\phi"] \arrow[r,"\alpha_M"] &p_X^* \mathcal{M} \arrow[d,"p_X^* \phi"] \\
&\sigma_X^* \mathcal{N}  \arrow[r,"\alpha_N"] &p_X^* \mathcal{N}. 
\end{tikzcd} 
\end{center}

We call such a morphism $G$-equivariant and denote the category of $G$-equivariant $\mathcal{O}_X$-modules together with $G$-equivariant morphisms by $\QCoh(\mathcal{O}_X,G)$.


\end{definition}

\begin{proposition}
\label{d1GequivariantAbeliancat}

Let $G$ an affine algebraic group acting on a scheme $X$ via $\sigma_X$. Then the category $\QCoh(\mathcal{O}_X,G)$ is Abelian.

\end{proposition}

\begin{proof}

By construction, we have that $\QCoh(\mathcal{O}_X,G)$ is additive. Let $(\mathcal{M}, \alpha_M)$ and $(\mathcal{N}, \alpha_N) \in \QCoh(\mathcal{O}_X,G)$ and let $\phi \in \Hom_{\QCoh(\mathcal{O}_X,G)}(\mathcal{M},\mathcal{N})$. Consider the exact sequence:

$$0 \to \ker(\phi) \to \mathcal{M} \to \mathcal{N} \to \coker(\phi) \to 0.$$

We aim to prove that $\ker(\phi)$ and $\coker(\phi)$ are in $\QCoh(\mathcal{O}_X,G)$. Since $\sigma_X$ and $p_X$ are smooth morphisms, the pullback functors $\sigma_X^*$ and $p_X^*$ are exact, so we get two short exact sequences:

\begin{equation}
\begin{split}
&0 \to \sigma_X^*\ker(\phi) \to  \sigma_X^*\mathcal{M} \to  \sigma_X^*\mathcal{N} \to  \sigma_X^*\coker(\phi) \to 0,\\
&0 \to p_X^*\ker(\phi) \to p_X^*\mathcal{M} \to p_X^*\mathcal{N} \to p_X^*\coker(\phi) \to 0.
\end{split}
\end{equation}

Consider now the diagram:

\begin{tikzcd}
&0 \arrow[r] &0 \arrow[r] &\sigma_X^*\ker(\phi) \arrow[r] &\sigma_X^*\mathcal{M} \arrow[d,"\alpha_{\mathcal{M}}"] \arrow[r] &\sigma_X^*\mathcal{N} \arrow[d,"\alpha_{\mathcal{N}}"]\\
&0 \arrow[r] &0 \arrow[r] &p_X^*\ker(\phi) \arrow[r] &p_X^*\mathcal{M} \arrow[r] &p_X^*\mathcal{N}.
\end{tikzcd}

By construction, we have that $\alpha_{\mathcal{M}}$ and $\alpha_{\mathcal{N}}$ are isomorphisms, so by the Five Lemma we obtain an isomorphism $\beta: \sigma_X^* \ker(\phi) \to p_X^* \ker(\phi)$. Furthermore, since $\alpha_{\mathcal{M}}$ and $\alpha_{\mathcal{N}}$ satisfy the cocycle condition, so does $\beta$. Thus, we have proven that $\ker(\phi) \in \QCoh(\mathcal{O}_X,G)$. A similar argument applying the Five Lemma shows that $\coker(\phi) \in \QCoh(\mathcal{O}_X,G)$. Finally, by construction we have that any monorphism  and epimorphism in $\QCoh(\mathcal{O}_X,G)$ is normal, so $\QCoh(\mathcal{O}_X,G)$ is indeed Abelian category.
\end{proof}

From now on, when we use the notion of morphism of $G$-equivariant $\mathcal{O}_X$-modules, we always view it as a morphism in the category $\QCoh(\mathcal{O}_X,G)$.

\textbf{A reformulation of equivariance}

We wish to reformulate the notion of an equivariant $\mathcal{O}$-module. Until the end of the section, we fix $X$ a scheme defined over $R$ acted on by an affine algebraic group $G$. For any $R$-algebra $A$, we define $X_A:=\Spec A \times_{\Spec R} X$. We start with a very simple observation: viewing $\mathcal{O}_X$ as a left $\mathcal{O}_X$-module, $(\mathcal{O}_X,\id)$ is a $G$-equivariant $\mathcal{O}_X$-module. We may reformulate this following ideas in \cite{MvdB}: for each $R$-algebra $A$ inducing a map $s:\Spec A \to \Spec R$ and for each geometric point $i_g:\Spec A \to G$ which induces an automorphism $g:X_{A} \to X_{A}$ there exists an isomorphism

                                                                    $$q_g: s^* \mathcal{O} \to (g^{-1})^*s^* \mathcal{O}, \text{ satisfying }$$

\begin{equation}
\label{d1structureequiveq}
              q_e=\id \text{ and } q_{gh}=(g^{-1})^*(q_h)q_g
\end{equation}              
in such a way that $(q_g)$'s are compatible with base change. Let $r_g=g^* \circ q_g$. For each $U \subset X_A$ affine open, $r_g$ induces a map $\mathcal{O}_{X_A}(U) \to \mathcal{O}_{X_A}(g^{-1}U)$. The equation \eqref{d1structureequiveq} translates as $r_e=\id$ and $r_{gh}=r_hr_g$. Furthermore, the $\mathcal{O}$-module compatibility requires that for any $f_1,f_2 \in \mathcal{O}_{X_A}(U)$, we have $r_g(f_1f_2)=r_g(f_1)r_g(f_2)$.  

We define $r_g$ via $r_g(f)(x)=f(g^{-1}x)$ for all $R$-algebras $A$, $U \subset X_A$ affine open, $x \in U$, $f \in \mathcal{O}_{X_A}(U)$, $g:X_{A} \to X_{A}$ and it is easy to see that $r_g$'s make $\mathcal{O}_X$ a $G$-equivariant $\mathcal{O}_X$-module according to equation \eqref{d1structureequiveq}. We may now make an abuse of notation: for each $i_g: \Spec A \to G$ and each $ f \in \mathcal{O}_{X_A}(U)$, we denote $g.f=r_{g^{-1}}(f)$ and we translate  the equivariance structure as

$$ e.f_1=f, \quad g.(h.f_1)=(gh).f_1, \quad g.(f_1f_2)=(g.f_1)(g.f_2) \text{ for all  } g,h \in G,  f_1,f_2 \in \mathcal{O}_X.$$

\begin{lemma}
A $\mathcal{O}_X$-module  $\mathcal{M}$ is $G$-equivariant if and only if for each $R$-algebra $A$, for each $s:\Spec A \to \Spec R$ and for each geometric point $i_g: \Spec A \to G$ which induces an automorphism $g:X_A \to X_A$ there exists an isomorphism of $\mathcal{O}_{A}$-modules

      $$ q_g:s^* \mathcal{M} \to (g^{-1})^* s^* \mathcal{M}        $$

satisfying

\begin{equation}
\label{d1weaklyequiveqOmod}
q_{e}=\id \text{ and }  q_{gh}=(g^{-1})^*(q_h)q_g
\end{equation}

in such a way that $(q_g)$'s are compatible with base change.

\end{lemma}

\begin{proof}

The proof repeats the argument in \cite[Proposition 2.2/Proposition 1.3.1]{MvdB} working over a commutative ring rather than a field and using the structure sheaf $\mathcal{O}$ instead of the sheaf of differential operators $\mathcal{D}.$
\end{proof}

Again, by setting $s_g=g^* \circ q_g$, we may reformulate equation \eqref{d1weaklyequiveqOmod} as:  for each $R$-algebra $A$ and for each $i_g:\Spec A \to G$, we have an isomorphism of $\mathcal{O}$-modules $s_g: \mathcal{M}_{X_A} \to \mathcal{M}_{X_A}$ such that for each $U \subset X_A$ affine open:

\begin{equation}
\label{d1alternativeeqOmod}
\begin{split}
&s_{e}=\id,\\
&s_{gh}=s_{h}s_{g},\\
&s_{g}'s \text{ are compatible with base change},\\
& r_g(f.m)=r_g(f).s_g(m) \text{, for all }  f \in \mathcal{O}_{X_A}(U), m \in \mathcal{M}(U).
\end{split}
\end{equation}

Again, we make an abuse of notation:  for each $i_g: \Spec A \to G$ and each $ m \in \mathcal{M}_A(U)$, we denote $g.m=s_{g^{-1}}(m)$ and we translate  the equivariance structure as:

\begin{equation}
\label{d1easyweakOmodequivariant}
\begin{split}
&e.m=m, \\
&gh.m= g.(h.m),\\
&g.(f.m)=(g.f).(g.m),
\end{split}
\end{equation}

for all $g,h \in G$, $m \in \mathcal{M}$,  $ f \in \mathcal{O}_X$.

Using the definition above, we reformulate the notion of $G$-equivariance of a morphism of $G$-equivariant $\mathcal{O}_X$-modules, $\phi: \mathcal{M} \to \mathcal{N}$ as:

\begin{equation}
\label{d1easymorphsimequivariance}
g.\phi(m)=\phi(g.m) \text{ for all } g \in G, m \in \mathcal{M}
\end{equation}

\section{Equivariant descent for \texorpdfstring{$\mathcal{O}$}{O}-modules}
\label{d1sectionequivariantdescentOmod}

\begin{definition}\cite[Section 4.3]{Annals}

Let $G$ be a smooth affine algebraic group of finite type, $Y$  a scheme equipped with an action $G \times Y \to Y$ and lastly, let  $X$ be a scheme. We say that a morphism $\xi: Y \to X$ is a $G$-torsor if $\xi$ is faithfully flat and locally of finite type, the action of $G$ respects $\xi$ and the map 

        $$G \times Y \to Y \times_{X} Y, \qquad (g,y) \mapsto (gy,y) $$

is an isomorphism.

An open subscheme $U$ of $X$ is said to trivialise the torsor $\xi$ if there is a $G$-invariant isomorphism
         $$ G \times U \to \xi^{-1}(U),$$
         
where $G$ acts on $G \times U$ by left multiplication on the first factor. 

Finally, let $\mathcal{S}_X$ be the set of affine open subschemes $U \subset X$ such that $U$ trivialises $\xi$ and $\mathcal{O}(U)$ is a finitely generated $R$-algebra. We say that $\xi$ is a locally trivial torsor if it can be covered by opens in $\mathcal{S}_X$.

\end{definition}

\begin{definition}[definition-proposition]
\label{d1pushforwardwelldefined}

Let $\xi:Y \to X$ be a locally trivial $G$-torsor and let $(\mathcal{M},\alpha_M)$ be a quasi-coherent $G$-equivariant $\mathcal{O}_Y$-module. Then the presheaf $(\xi_* \mathcal{M})^G$ acquires the structure of a quasi-coherent $\mathcal{O}_X$-module. Furthermore, if we are given $\psi: (\mathcal{M},\alpha_M) \to (\mathcal{N},\alpha_N)$ a map of $G$-equivariant $\mathcal{O}_Y$-modules there is a canonical induced map $(\xi_*)^G\psi: (\xi_* \mathcal{M})^G \to (\xi_* \mathcal{N})^G$.

\end{definition}

\begin{proof}

The  question is local, so we may assume that $X$ is affine, $\xi: Y \to X$ is $p_X: G \times X \to X$ and $G$ acts on $G \times X$ via left multiplication on the first factor. Since $G$ and $X$ are affine, the category of $G$-equivariant $\mathcal{O}_{G \times X}$-modules is equivalent to the category of $( \mathcal{O}(G \times X),G)$-modules by the same arguments as in the proof of \cite[Proposition 1.4.1]{MvdB}. Modules in this category are modules equipped with compatible actions of the ring $\mathcal{O}(G \times X)$ and of the group $G$.

Let $\mathcal{M}$ be a $G$-equivariant $\mathcal{O}_{G \times X}$-module and let $M=\Gamma(G \times X,\mathcal{M})$ be its global sections. Since $M$ is a $( \mathcal{O}(G \times X),G)$-module it acquires a comodule structure $\rho_M: M \to \mathcal{O}(G) \uset{R} M$. Furthermore, let $\rho_G: \mathcal{O}(G \times X) \to \mathcal{O}(G) \uset{R} \mathcal{O}(G \times X)$ denote the comodule structure on $\mathcal{O}(G \times X)$ induced from the $G$-action on $G \times X$ by left multiplication on the first factor. 

As $X$ is affine, to prove the first statement, it is enough to prove that the Abelian group $(p_* \mathcal{M})^G(X)=M^G$ has the structure of an $\mathcal{O}(X)$-module.

Let $f \in \mathcal{O}(X)$ and define $\phi \in \mathcal{O}(G \times X)$ by $\phi(h,x):= f(x) \text{ for all } h \in G, x \in X,$ and for any $m \in M^G$ define
$ f . m := \phi . m$

We need to prove that $f.m \in M^G$. By construction, it is clear that $\rho_G (\phi) = 1 \otimes \phi$. Since $M$ is a $(\mathcal{O}(G \times X),G)$-module, we  have

$$\rho_M ( \phi. m)= \rho_G(\phi) \rho_M(m)= (1 \otimes \phi) (1 \otimes m)= 1 \otimes \phi. m,$$

so $\phi.m \in M^G$, thus the action of $f$ is well-defined.

For the second statement, notice that if $\varphi: (M,\alpha_M) \to (N,\alpha_N)$ is a morphism of $( \mathcal{O}(G \times X),G)$-modules, it is $G$-equivariant, so $\varphi$ restricts to a map $M^G \to N^G$. 

In general if $\xi: \mathcal{M} \to \mathcal{N}$ is a map of $G$-equivariant $\mathcal{O}_Y$ modules and we let $U$ be affine open in $X$, we have canonical maps  $(\xi_*)^G\psi: (\xi_* \mathcal{M})^G(U) \to (\xi_* \mathcal{N})^G(U)$ compatible with restrictions given by restricting the map $\psi$. Therefore, glueing together the local morphisms  we get a map of sheaves.
\end{proof}

In particular, we have proven that if $\xi:Y \to X$ is a locally trivial $G$-torsor, we obtain a functor $\xi_*^G$ from $G$-equivariant $\mathcal{O}_Y$-modules to $\mathcal{O}_X$-modules. We would like to prove that this is an equivalence of categories. 

Recall that for an  $R$-Hopf algebra $H$, a Hopf module $M$ is a left $H$-module, together with a comodule map $\rho:M \to H \uset{R} M$ such that $\rho$ is a map of $H$-modules; here we view $H$ as a module over itself via left multiplication.

\begin{theorem}\cite[Theorem 4.13]{FundHopf}
\label{d1fundhopftheorem}
Let $H$ be a Hopf algebra over a commutative ring $R$ and $M$ a Hopf module. Denote $M^{\co H}$ the coinvariants of $M$. Then the map
                 
               $$\mu: H \uset{R} M^{\co H} \to M, \quad \mu(h \otimes m)=h.m$$
is an isomorphism.               
\end{theorem}

We may now prove the main result of this section.
\begin{proposition}
\label{d1Omodulesequivariantdescent}[Equivariant descent for $\mathcal{O}$-modules]
Let $G$ be a smooth affine algebraic group of finite type and let $\xi:Y \to X$ be a locally trivial $G$-torsor. Then the functors $\xi_* (-)^G$ and $\xi^*(-)$ induce quasi-inverse equivalences of categories between $G$-equivariant quasi-coherent $\mathcal{O}_Y$-modules and quasi-coherent $\mathcal{O}_X$-modules.

\end{proposition}

\begin{proof}

For $\mathcal{M} \in \QCoh(\mathcal{O}_Y,G)$ and $\mathcal{N} \in \QCoh(\mathcal{O}_X)$ we obtain by functoriality maps $\mathcal{M} \to \xi^*(\xi_* \mathcal{M})^G $ and $(\xi_*(\xi^* \mathcal{N}))^G \to \mathcal{N}$, respectively. Thus we only need to prove the statement locally. We may then assume that $Y=G \times X$, $p:G \times X \to X$ is the projection onto the second factor and $G$ acts on $G \times X$ via left multiplication on the first factor.

We start by constructing a natural isomorphism $\eta:\id \to (p_*)^Gp^*$. Let $\mathcal{M}$ be a $\mathcal{O}_X$-module. We aim to define a map $\eta_M: \mathcal{M} \to p_*^Gp^* \mathcal{M}$.  For any open affine $U \subset X$, we have
  
$$(p_*^Gp^* \mathcal{M})(U)= (\mathcal{O}(G) \uset{R} \mathcal{M}(U))^G.$$ 

Let $\eta_{\mathcal{M}U}: \mathcal{M}(U) \to (\mathcal{O}(G) \uset{R} \mathcal{M}(U))^G$ be defined by $m \mapsto 1 \otimes m$.


Since $G$ is flat over $R$, $\mathcal{O}(G)$ is a flat $R$-module, so a direct limit of free $R$-modules. Write $\mathcal{O}_G= \lim M_i$, where $M_i$ is a free $R$-module. We have that  $(M_i \otimes_R \mathcal{M}_U)^G \cong M_i$ and since rational cohomology commutes with direct limits by \cite[Lemma I.4.17]{Jan1}, we obtain that $\eta_{\mathcal{M}U}$ is indeed an isomorphism. 

Let $V \subset U$ be open affine and let $\res_{UV}: \mathcal{M}(U) \to \mathcal{M}(V)$ be the restriction map.  It is easy to see that the following diagram is commutative:

\begin{center}
\begin{tikzcd}
& \mathcal{M}(U) \arrow[d,"\res_{UV}"] \arrow[r,"\eta_{\mathcal{M}U}"] & (\mathcal{O}(G) \uset{R} \mathcal{M}(U))^G \arrow[d,"\id \uset{R} \res_{UV}"] \\
& \mathcal{M}(V) \arrow[r,"\eta_{\mathcal{M}V}"] & (\mathcal{O}(G) \uset{R} \mathcal{M}(V))^G.

\end{tikzcd}
\end{center}

Thus, $\eta$ is a map of sheaves. Next, we prove that the isomorphism $\eta$ is natural. Let $\varphi: \mathcal{M} \to \mathcal{N}$ be a morphism of $\mathcal{O}_X$-modules. It is enough  to show that the  following diagram is commutative:

\begin{center}
\begin{tikzcd}
 
&\mathcal{M} \arrow[d,"\varphi"]  \arrow[r,"\eta_M"]  &(p_*)^Gp^* \mathcal{M} \arrow[d,"(p_*)^Gp^* \varphi"] \\
&\mathcal{N} \arrow[r,"\eta_N"] & (p_*)^Gp^* \mathcal{N}.

\end{tikzcd}
\end{center}

We can work locally. Let $U \subset X$ be affine open and let $m \in \mathcal{M}(U)$. Then $\eta_N (\varphi(m))= 1 \otimes \varphi(m)$. On the other hand we have $p^* \varphi: p^*\mathcal{M}( G \times U) \to p^*\mathcal{N}(G \times U)$ defined by $p^*\varphi(F \otimes m)= F \otimes \varphi(m)$ for any $F \in \mathcal{O}(G), m \in \mathcal{M}(U)$.

Thus, we have that $(p_*)^G p^* \varphi:(p_*)^Gp^* \mathcal{M}(U) \to (p_*)^Gp^* \mathcal{M}(U) $ is defined by $(p_*)^G p^* \varphi (1 \otimes m)= 1 \otimes \varphi(m)$, for all $m \in \mathcal{M}(U)$. In particular, we get that 

$$(p_*)^Gp^* \varphi (\eta_M(m))= (p_*)^Gp^* \varphi (1 \otimes m)= 1 \otimes \varphi(m)= \eta_N (\varphi(m)),$$

which shows that the diagram is indeed commutative.

Now, let $(\mathcal{M},\alpha)$ be a $G$-equivariant $\mathcal{O}_{G \times X}$-module. By construction $(p_*)^G \mathcal{M}$ is a subsheaf of $p_* \mathcal{M}$ and since there is a canonical sheaf map $p^*(p_*) \mathcal{M} \to \mathcal{M}$, we get by functoriality that there is a map $ \nu_{\mathcal{M}}:p^*(p_*)^G \mathcal{M} \to \mathcal{M}$.

Let $U \subset X$ open affine. Then we have the induced map $$\nu_{\mathcal{M}_{G \times U}}:\mathcal{O}(G) \uset{R} \mathcal{M}(G \times U)^G = p^*(p_*)^G \mathcal{M} (G \times U) \to \mathcal{M}(G \times U) $$

given by $\nu_{\mathcal{M}_{G \times U}}( f \otimes m)= f.m$. We aim to prove that this map is an isomorphism. Since $\mathcal{M}$ is $G$-equivariant the isomorphism $\alpha$ induces an automorphism of $\mathcal{O}_{G \times G \times X}$-modules on $p^* \mathcal{M}$, so in particular we obtain an automorphism on $\mathcal{O}(G) \otimes \mathcal{M}(G \times U)$ of $\mathcal{O}(G \times G \times U)$-modules. This induces a Hopf module structure on $\mathcal{M}(G \times U)$ for the Hopf algebra $\mathcal{O}(G)$. Thus, by Theorem \ref{d1fundhopftheorem} we get that $\nu_{\mathcal{M}_{G \times U}}$ is indeed an isomorphism.

Let $\{U_i\}_{i \in I}$ be an affine open cover of $X$ and let $\{G \times U_i\}_{i \in I}$ be the corresponding affine cover of $G \times X$. As $\nu_{\mathcal{M}_{G \times U_i}}$ is an isomorphism, the sheaves $p^*(p_*)^G \mathcal{M}$ and $\mathcal{M}$ agree on an affine open cover and there is a sheaf map between the two, $\nu$ is an isomorphism.

To finish the proof, notice that by construction, we have $\nu:p^*(p_*)^G \to \id$ is a natural isomorphism. This concludes the proof of the proposition.
\end{proof}

We will also consider a slightly more general setting. Let $Y$ be a variety acted on by two smooth affine algebraic groups of finite type, $G$ and $B$. Let us denote the two actions by $.$ and $\star$.

\begin{observation}
\label{d1easyobservation}
Let $\mathcal{M}$ be a $G \times B$-equivariant $\mathcal{O}_Y$-module. Then $\mathcal{M}^G$ is a $B$-equivariant submodule of $\mathcal{M}$. Let $\phi: \mathcal{M} \to \mathcal{M}'$ be a $G \times B$ equivariant map of $G \times B$-equivariant modules. Then $\phi$ restricts to a $B$-equivariant map $\phi:\mathcal{M}^G \to \mathcal{M}'^G$.

\end{observation}

\begin{proof}
We view the equivariance structure via the equations \eqref{d1easyweakOmodequivariant}.  We have that for $g \in G, b \in B $ and $m \in \mathcal{M}$ that

$$g.(b \star m)= b \star (g.m).$$

If $m \in \mathcal{M}^G$, then $g.m=m$, so by the equation above $g.(b \star m)= b \star m$, so $b \star m \in \mathcal{M}^G$. Thus, the $B$-equivariance on $\mathcal{M}$ induces $B$-equivariance on $\mathcal{M}^G$.

Similarly, using the equivariance of morphisms in equation \eqref{d1easymorphsimequivariance}, we have since $\phi$ is particular $G$-equivariant that for $m \in \mathcal{M}^G$, $g.\phi(m)=\phi(g.m)=\phi(m)$, so $\phi$ restricts to a map $\mathcal{M}^G \to \mathcal{M}'^G$.

Finally, since $\phi$ is in particular $B$-equivariant, we have that for $m \in \mathcal{M}^G$

 $$b \star \phi(m)= \phi( b \star m),$$
 
concluding the proof. 
\end{proof}

\begin{lemma}
\label{d1Omoduledescentpreserveequivariance}
Let $G$ and $B$ be smooth affine algebraic groups of finite type acting on $R$-schemes $X$ and $Y$ such that the action of $B$ and $G$ on $Y$ commute. Let $\xi:Y \to X$ be a locally trivial $G$-torsor that is $B$-equivariant.
 
\begin{itemize}

\item {Let $\mathcal{M}$ be a $G \times B$-equivariant $\mathcal{O}_Y$-module. Then $(\xi_* \mathcal{M})^G$ is a $B$-equivariant $\mathcal{O}_X$-module.}
\item {Let $\mathcal{N}$ be a $B$-equivariant $\mathcal{O}_X$-module. Then $\xi^* \mathcal{N}$ is a $G \times B$-equivariant $\mathcal{O}_Y$-module.}
\end{itemize} 

\end{lemma}

\begin{proof}
Since $\xi$ is $B$-equivariant, the $B$-action on $\mathcal{O}_X \cong (\xi_* \mathcal{O}_Y)^G$ is induced from the $B$-action on $\mathcal{O}_Y$. Further, using the observation above we may define a $B$-action on $(\xi_* \mathcal{M})^G$ which is compatible with the $B$-action on $\mathcal{O}_Y$ since $\mathcal{M}$ is $B$-equivariant, so the first claim is proven. 

For the second claim, we let $G$ act on $\mathcal{N}$ via $g . n=n$ for all $g \in G$ and $n \in \mathcal{N}$, so that $\mathcal{N}$ is $G \times B$-equivariant. The claim follows from Lemma \ref{d1Oequivpreservefunctor}.
\end{proof}

\begin{corollary}
\label{d1BequivarianteqdescentOmodules}
Let $G$ and $B$ be smooth affine algebraic groups of finite type acting on $Y$ and $X$ such that the action of $B$ and $G$ on $Y$ commute. Let $\xi:Y \to X$ be a locally trivial $G$-torsor that is $B$-equivariant. The functors $\xi_* (-)^G$ and $\xi^*(-)$ induce quasi-inverse equivalences of categories between $G \times B$-equivariant quasi-coherent $\mathcal{O}_Y$-modules and quasi-coherent $B$-equivariant $\mathcal{O}_X$-modules.

\end{corollary}

\begin{proof}
This follows from Proposition \ref{d1Omodulesequivariantdescent},  Lemma \ref{d1Omoduledescentpreserveequivariance} and Observation \ref{d1easyobservation}.
\end{proof}

\section{Deformed twisted differential operators}
\label{d1sectiontdo's}

\begin{definition}
\label{d1definitionofRvariety}
We call an $R$-scheme $X$ that is smooth, separated and locally of finite type an \emph{$R$-variety}.

\end{definition}
 We write $\mathcal{T}_X$ for the sheaf of sections of the tangent bundle $TX$.
\begin{definition}\cite[Definition 4.2]{Annals}
\label{d1crystallinedifferentialoperatorsdefinition}

Let $X$ be an $R$-variety. The sheaf of crystalline differential operators is defined to be the enveloping algebra $\mathcal{D}_X$ of the Lie algebroid $\mathcal{T}_X$.
\end{definition}

We can view $\mathcal{D}_X$  as a sheaf of ring generated by $\mathcal{O}_X$ and $\mathcal{T}_X$ modulo the relations:

\begin{itemize}
\item{$f \partial =f \cdot \partial$;}
\item{$ \partial f - f \partial= \partial (f)$;}
\item{$ \partial \partial' - \partial' \partial=[\partial,\partial'],$}
\end{itemize}

for all $f \in \mathcal{O}_X$ and $\partial,\partial' \in \mathcal{T}_X$. The sheaf $\mathcal{D}_X$ comes equipped with a natural PBW filtration:

 $$0 \subset F_0 (\mathcal{D}_X) \subset  F_1 (\mathcal{D}_X) \subset \ldots $$

consisting of coherent $\mathcal{O}_X$-modules such that

$$F_0(\mathcal{D}_X)= \mathcal{O}_X, \quad F_1( \mathcal{D}_X)= \mathcal{O}_X \oplus \mathcal{T}_X, \quad F_m(\mathcal{D}_X) = F_1 (\mathcal{D}_X) \cdot F_{m-1}( \mathcal{D}_X)  \text{ for } m >1.                   $$

Since $X$ is smooth, the tangent sheaf $\mathcal{T}_X$ is locally free and the associated graded sheaf of algebras of $\mathcal{D}_X$ is isomorphic to the symmetric algebra of $\mathcal{T}_X$:

\begin{equation}
\label{d1gradingcrysdifop}
 \gr(\mathcal{D}_X)= \bigoplus_{m=1}^{\infty} \frac{F_m(\mathcal{D}_X)}{F_{m-1}(\mathcal{D}_X)} \cong \Sym_{\mathcal{O}_X} \mathcal{T}_X.
\end{equation}

If $q:T^*X \to X$ is the cotangent bundle of $X$ defined by the locally free sheaf $\mathcal{T}_X$, then we can also identify $\gr(\mathcal{D}_X)$ with $q_{*}\mathcal{O}_{T^*X}$.

Let $X$ be an $R$-variety and let $U=\Spec(A) \subset X$ be open affine. Further, we consider $\mathcal{M}$ a sheaf of $\mathcal{O}_X$-bimodules quasi-coherent with respect to the left action. We define a filtration on $M=\mathcal{M}(U)$ given by $F_{\bullet}M$:

\begin{itemize}
\item{$F_{-1}(M)=0,$}
\item{$F_n(M)= \{ m \in M | \ad(a_0)\ad(a_1)...\ad(a_n)(m)=0 \text{, for any } a_0,a_1, \ldots a_n \in A\}$}, for $n \geq 0$.

\end{itemize} 

We say that $M$ is differential if $M= \cup_{n \in \mathbb{N}} F_n(M)$ and we call $\mathcal{M}$ a differential $\mathcal{O}_X$-bimodule if there is an affine open cover $(U_i)_{i \in I}$ such that $\mathcal{M}(U_i)$ is a differential bimodule for all $i \in I$.

Let $\mathcal{M},\mathcal{N}$ be two quasi-coherent $\mathcal{O}_X$-modules. Then for any affine open $U$ in $X$, the set $\Hom_{R}(\mathcal{M}(U),\mathcal{N}(U))$ has the structure of a $\mathcal{O}_X(U)$-bimodule. Let $\mathcal{F} \in \Hom_{R}(\mathcal{M}, \mathcal{N}$); we say that $\mathcal{F}$ is a differential operator of degree $\leq n$ if for any affine open $U$, $\mathcal{F}(U) \in F_n(\Hom_{R}(\mathcal{M}(U),\mathcal{N}(U)).$ We denote $\Dif^n(\mathcal{M},\mathcal{N})$ the subsheaf of differential operators of degree $\leq n$ and $\Dif(\mathcal{M},\mathcal{N})=\cup_{n \in \mathbb{N}}  \Dif^n(\mathcal{M},\mathcal{N})$ the subsheaf of differential operators of finite degree. 

We may construct differential $\mathcal{O}_X$-modules using the following proposition:

\begin{proposition}
Let $\mathcal{M}$ be a coherent $\mathcal{O}_X$-module and let $\mathcal{N}$ be a $\mathcal{O}_X$-module. Then $\Dif(\mathcal{M},\mathcal{N})$ is differential $\mathcal{O}_X$-bimodule.

\end{proposition}

\begin{proof}
The proof follows by repeating the argument in \cite[Proposition 2.1.3]{Wan}.
\end{proof}

\begin{definition}
Let $\mathcal{A}$ be a $\mathcal{O}_X$-algebra. We say that $\mathcal{A}$ is a differential algebra if $\mathcal{A}$ is a flat $R$-module and multiplication makes $\mathcal{A}$ a differential $\mathcal{O}_X$-bimodule. The filtration $F_{\bullet}(\mathcal{A})$ becomes a ring filtration and with respect to this filtration $\gr^{F}(\mathcal{A})$ is commutative.

\end{definition}


\begin{definition}
\label{d1tdodefinition}
Let $r \in R$ be a regular element. An algebra of $r$-deformed twisted differential operators(tdo) is an $\mathcal{O}_X$-differential algebra $\mathcal{D}$ such that:
\begin{enumerate}[label=\roman*)]
\item{ The natural map $\mathcal{O}_X \to F_0(\mathcal{D})$ is an isomorphism.}
\item{ The morphism $\gr_{1}^F \mathcal{D} \to \mathcal{T}_X=\Der_R(\mathcal{O}_X,\mathcal{O}_X)$ defined by $\psi \mapsto \ad_{\psi}$ for $\psi \in F_1(\mathcal{D})$ induces an isomorphism $\gr_{1}^F \mathcal{D} \cong r\mathcal{T}_X$.}
\item{The morphism of $\mathcal{O}_X$-algebras $\Sym_{\mathcal{O}_X}(\gr_1^F \mathcal{D}) \to \gr^F \mathcal{D}$ is an isomorphism.}
\end{enumerate}

A morphism of tdo's is a morphism of $\mathcal{O}_X$ algebras compatible with the $\mathcal{O}_X$-bimodule structure and the maps in $i) \to iii)$.
\end{definition}

We should make some remarks about this definition: when $r=1$ we call $\mathcal{D}$ a sheaf of twisted differential operators. Classically, working with twisted differential operators over a complex variety the condition $iii)$ is implied by $i)$ and $ii)$. This is no longer true in our case. Further, the sheaf of \emph{Grothendieck}'s differential operators does not satisfy condition $iii)$ for a general ring $R$. This is the main reason why we develop the theory of twisted differential operators using the connection with Lie algebroids rather than using the classical $\Dif$ definition. 

\begin{lemma}
\label{d1crystallinedifoplocnoether}
Assume that $X$ is locally Noetherian $R$-variety and let $\mathcal{D}$ be an $r$-deformed tdo on $X$. Then $\mathcal{D}$ is locally Noetherian. 

\end{lemma}

\begin{proof}

We have by conditions $ii)$ and $iii)$ that $\gr^{F} \mathcal{D} \cong \Sym_{\mathcal{O}_X}(r\mathcal{T}_X)$ and because $r$ is regular $\Sym_{\mathcal{O}_X}(r\mathcal{T}_X) \cong \Sym_{\mathcal{O}_X}(\mathcal{T}_X)$. Since $\mathcal{T}_X$ is a free $\mathcal{O}_X$-module and $X$ is locally Noetherian, we obtain that $\Sym_{\mathcal{O}_X}(\mathcal{T}_X)$ is locally Noetherian. Therefore, we have $\gr^{F} \mathcal{D}$ is locally Noetherian, which implies the same for $\mathcal{D}$. 
\end{proof}

\section{\texorpdfstring{Connections between deformed Lie algebroids \\ and deformed  tdo's}{Connections between deformed Lie algebroids and deformed  tdo's}}
\label{d1sectionliealgtdo}

Throughout this section, we let $X$ denote an $R$-variety and $r \in R$ a regular element.

\begin{definition}
A Lie algebroid $\mathcal{L}$ on $X$ is a quasi-coherent $\mathcal{O}_X$-module equipped with a morphism of $\mathcal{O}_X$-modules $\rho:\mathcal{L} \to \mathcal{T}_X$ (the anchor map) and an $R$-linear pairing $[\bullet,\bullet]:\mathcal{L} \to \mathcal{T}_X$ such that:

\begin{itemize}
\item{$[\bullet,\bullet]$ defines the structure of a Lie algebra on $\mathcal{L}$ and $\rho$ is a morphism of Lie algebras.}
\item{$[l_1,fl_2]=f[l_1,l_2]+\rho(l_1)(f)l_2$ for $l_i \in \mathcal{L}, f \in \mathcal{O}_X$.}
\end{itemize}

A morphism of Lie algebroids is a morphism of $\mathcal{O}_X$-modules compatible with the anchor maps  and bracket.
\end{definition}

In particular, locally we obtain that for any $U \subset X$ affine open that $\mathcal{L}(U)$ is an $(R,\mathcal{O}_X(U))$-Lie Rinehart algebra, see \cite{Rinehart} for definition and basic properties of Lie Rinehart algebras. We may think of $\mathcal{L}$ as a sheaf of $(R,\mathcal{O}_X)$-Lie Rinehart algebras; we will use this local description soon.

\begin{definition}
The universal enveloping algebra of $\mathcal{L}$, denoted $U(\mathcal{L})$, is the sheaf of $R$-algebras generated by $\mathcal{O}_X$ and $\mathcal{L}$ modulo the relations:

\begin{itemize}
\item{$i:\mathcal{O}_X \to U(\mathcal{L})$ is a morphism of $R$-algebras,}
\item{$j:\mathcal{L} \to U(\mathcal{L})$ is a morphism of Lie algebras,}
\item{$j(fl)=i(f)j(l)$ and $[j(l),i(f)]=i(\rho(l)(f))$.}
\end{itemize}

\end{definition}

Locally, $U(\mathcal{L})$ is just the enveloping algebra of the corresponding $(R,\mathcal{O}_X(U))$-algebra.

We want to establish a correspondence between Lie algebroids and $r$-deformed tdo's on an $R$- variety $X$. For an $\mathcal{O}_X$-differential algebra $\mathcal{D}$ we define $\Lie(\mathcal{D}):=F_1(\mathcal{D})$; one may prove that when $\mathcal{D}$ is a tdo, $\Lie(\mathcal{D})$ is a Lie algebroid, see \cite[1.2.5]{BB2}; unfortunately not all Lie algebroids induce tdo's, so we need a more specific notion. 

\begin{definition}
We call a Lie algebroid $\mathcal{L}$ an $r$-deformed Picard algebroid if there exists a short exact sequence of Lie algebras and $\mathcal{O}_X$-modules:

$$0 \to \mathcal{O}_X \to \mathcal{L} \to r\mathcal{T}_X \to 0.$$

One should notice that the Lie algebra structure imposed on $\mathcal{O}_X$ is the trivial one. We should also denote $1_{\mathcal{L}}$ the image of $1 \in \mathcal{O}_X$ under the inclusion map.

A morphism of $r$-deformed Picard algebroids is a morphism of Lie algebroids compatible with the maps in the short exact sequence defining $r$-deformed Picard structure.
\end{definition}

\begin{proposition}
\label{d1liealgebroidtotdo}
Let $\mathcal{L}$ be an $r$-deformed Picard algebroid on $X$. Then the sheaf of rings  $\mathcal{D}:=U(\mathcal{L})/U(\mathcal{L})(i(1)-j(1))$ is an $r$-deformed tdo with $\Lie(\mathcal{D})=\mathcal{L}$.

\end{proposition}

\begin{proof}
The question is local so we may assume that $X$ is affine; $A=\mathcal{O}_X(X)$, $T=\mathcal{T}_X(X)$, $L=\mathcal{L}(X)$ and let $i:A \to L$ denote the injection induced by the short exact sequence defining $\mathcal{L}$.

Consider the enveloping algebra $U(L)$ of the $(R,A)$-Lie algebra $L$. We can think of it as being generated by $A$ and the universal enveloping algebra of the Lie algebra $L$ subject to the relations: $fl=f \cdot l$ (the module action), $fl-lf=\rho(l)(f)$ for $f \in A$ and $l \in L$, $l_1l_2-l_2l_1=[l_1,l_2]$ for $l_1,l_2 \in L$. The natural filtration on $U(L)$ is given by:

\begin{itemize}
\item{$F_0(U(L))=A,$}
\item{For $n \geq 1$, $F_n(U(L))=A+l_1 l_2 \ldots l_m $, where $m \leq n$ and $l_1,l_2,\ldots l_m \in L$.}
\end{itemize}

Since $fl-lf=\rho(l)(f)$ for $f \in A$ and $l \in L$, it easy to see that $U(L)$ becomes a differential $A$ algebra with respect to this filtration. 

Let $I=\langle i(1)-1_L \rangle$ be the central two sided ideal and $D:=\mathcal{D}(X)=U(L)/IU(L)$. We give $D$ the quotient natural quotient filtration induced from $U(L)$. By construction, we have that $D$ is also a differential $A$-algebra since $U(L)$ is. Furthermore, we have $F_0(D)\cong F_0(U(L)) \cong A$.

Let $g:F_1(U(L)) \to L$ given by $g(a+l)=i(a)+l$ for $a \in A, l \in L$. Then it is clear that $g$ is surjective and $\ker(g)=a-i(a)=F_1(I)$. Thus, we obtain $F_1(D) \cong L$, so $\Lie(D) \cong L$.

Since $L$ is an $r$-deformed Picard algebroid, we obtain immediately that $\gr_1 D \cong L/A \cong rT$. Finally, because $X$ is an $R$-variety, $T$ is projective as an $A$-module. Therefore, $L$ is also projective as $A$-module, so we have by \cite[Theorem 3.1]{Rinehart} that $\gr D \cong \Sym_A(\gr_1 D).$ Thus, we have proven all the required properties to make $D$ an $r$-deformed tdo on $X$.
\end{proof}

Using the Lemma above we make the following definition:
\begin{definition}

Let $X$ be an $R$-variety. Define a functor $\mathscr{T}:$ category of $r$-deformed Picard algebroids on $X \to$ category of $r$-deformed twisted differential operators on $X$ by

$$\mathscr{T}(\mathcal{L}):= U(\mathcal{L})/U(\mathcal{L})(i(1)-j(1)).$$

\end{definition}

\begin{lemma}
\label{d1tdotoliealgebroid}

Let $\mathcal{D}$ be an $r$-deformed tdo on $X$. The sheaf $\Lie(\mathcal{D}):=F_1(\mathcal{D})$ is an $r$-deformed Picard algebroid and furthermore $$\mathscr{T}(\Lie \mathcal{D})=U(\Lie(\mathcal{D}))/U(\Lie(\mathcal{D}))(i(1)-j(1)) \cong \mathcal{D}.$$

\end{lemma}

\begin{proof}

Let $\mathcal{L}:=\Lie(\mathcal{D})$. Since $\mathcal{D}$ is a differential algebra, $\mathcal{L}$ is a Lie algebroid by \cite[1.2.5]{BB2}, with the anchor map  $\rho:\mathcal{L} \to \mathcal{T}_X$ induced  by axiom $ii)$ of Definition \ref{d1tdodefinition}. Further by axioms $i)$ and $ii)$ of \ref{d1tdodefinition} we observe that  $\ker(\rho)=\mathcal{O}_X$ and $\im(\rho)=r\mathcal{T}_X$, so $\mathcal{L}$ is indeed an $r$-deformed Picard algebroid.

Let $\mathcal{A}:=U(\mathcal{L})/U(\mathcal{L})(i(1)-j(1))$. By Proposition \ref{d1liealgebroidtotdo}, $\mathcal{A}$ is an $r$-deformed tdo. Further by construction we have that there is morphism of filtered algebras $\mathcal{A} \to \mathcal{D}$ and $\gr_1(\mathcal{A}) \cong \gr_1(\mathcal{D})$. Since $\mathcal{A}$ and $\mathcal{D}$ are $r$-deformed tdo's, we have by axiom $iii)$ of Definition \ref{d1tdodefinition} that  $\gr(\mathcal{A}) \cong \Sym_{\mathcal{O}_X}(\gr_1 \mathcal{A})$ and $\gr(\mathcal{D}) \cong \Sym_{\mathcal{O}_X}(\gr_1 \mathcal{D} )$. Therefore, we get $\gr(\mathcal{A}) \cong \gr(\mathcal{D})$,so $\mathcal{A} \cong \mathcal{D}$ since there is a filtered morphism between them.

\end{proof}

\begin{corollary}
\label{d1Picardalgtdocorresp}
Let $X$ be an $R$-variety. The functors $\mathscr{T}$ and $\Lie$ induce quasi-inverse equivalences of categories between the category of $r$-deformed Picard algebroids on $X$ and the category of $r$-deformed tdo's on $X$.

\end{corollary}

\begin{proof}
This follows from Proposition \ref{d1liealgebroidtotdo} and Lemma \ref{d1tdotoliealgebroid}.
\end{proof}

\section{Equivariant deformed Picard algebroids and deformed homogeneous twisted differential operators}
\label{d1eqliealgandhtdo}

Throughout this section, we fix $X$ an $R$-variety, $G$ a smooth affine algebraic group of finite type acting on $X$ and $r \in R$ a regular element. Recall from \ref{d1easyweakOmodequivariant} that for a $G$-equivariant $\mathcal{O}_X$-module $\mathcal{M}$ we denoted by abuse of notation by $\cdot$ the group action giving the equivariance.
\begin{definition}
\label{d1defofeqalg}
Let $(\mathcal{L},\rho)$ be a Lie algebroid. We say that $\mathcal{L}$ is $r$-deformed $G$-equivariant if $\mathcal{L}$ is a $G$-equivariant as a $\mathcal{O}_X$-module and it is equipped with a Lie algebra morphism $i_{\mathfrak{g}}:r\mathfrak{g} \to \mathcal{L}$ such that:

\begin{enumerate}[label=\roman*)]
\item{$g \cdot[x,y]=[g \cdot x, g \cdot y]$, for $g \in G$, $x,y \in \mathcal{L}$.}
\item{$g.\rho(l)(f)=\rho(g.l)(g.f)$, for $g \in G, l \in \mathcal{L}, f \in \mathcal{O}_X.$ This is equivalent to $\rho$ being $G$-equivariant.}
\item{$i_{\mathfrak{g}}(g.\psi)=g.i_{\mathfrak{g}}(\psi)$, for $g \in G$ and $\psi \in r\mathfrak{g}$. Here $G$ acts on $r\mathfrak{g} \subset \mathfrak{g}$ via the Adjoint action.}
\end{enumerate}
\end{definition}

Similarly, we may define the notion of equivariant differential algebra.

\begin{definition}
\label{d1defequivdiffalg}
Let $\mathcal{D}$ be a differential $\mathcal{O}_X$-algebra. We call $\mathcal{D}$ an $r$-deformed \emph{$G$-equivariant} differential algebra if it is $G$-equivariant as a left $\mathcal{O}_X$-module and it is equipped with a Lie algebra map $i_{\mathfrak{g}}:r\mathfrak{g} \to \mathcal{D}$ such that:

\begin{enumerate}

\item{$g.1=1$ and $g.(d_1d_2)=(g.d_1)(g.d_2)$, for $g \in G$ and $d_1,d_2 \in \mathcal{D}.$ }
\item{$g.(fd)=(g.f)(g.d)$, for $f \in \mathcal{O}_X$ and $d \in \mathcal{D}$.}
\item{$i_{\mathfrak{g}}(g.\psi)=g.i_{\mathfrak{g}}(\psi)$, for $g \in G$ and $\psi \in r \mathfrak{g}$.}
\end{enumerate}

\end{definition}

\begin{lemma}
\label{d1equivariantdifalg}
Let $(\mathcal{L},\rho,i_{\mathfrak{g}})$ be an $r$-deformed $G$-equivariant Lie algebroid. Then $U(\mathcal{L})$ is an $r$-deformed $G$-equivariant differential algebra.
\end{lemma}

\begin{proof}


Since $\mathcal{L}$ is quasi-coherent as a $\mathcal{O}_X$-module, so is $U(\mathcal{L})$. We define $G$-action on $U(\mathcal{L})$ by defining $g.(l_1l_2 \ldots g.j)=(g.l_1)(g.l_2) \ldots (g.l_j)$ for $l_1,l_2, \ldots l_j \in \mathcal{L}$ and $g.(fl_1l_2 \ldots l_j)=(g.f)(g.l_1l_2 \ldots l_j)$ for $f \in \mathcal{O}_X$ and $l_1,l_2, \ldots l_j \in \mathcal{L}$. We have

\begin{equation}
\begin{split}
g.(l_1l_2-l_2l_1) &= (g.l_1)(g.l_2)-(g.l_2) (g.l_1) \\
                  &=[g.l_1,g.l_2] \\
                  &=g.[l_1,l_2]
\end{split}
\end{equation}
and 

\begin{equation}
\begin{split}
g.(fl-lf) &= (g.f)(g.l)-(g.l)(g.f)\\
          &=[g.f,g.l] \\
          &=\rho(g.l)(g.f)\\
          &=g.\rho(l)(f).
\end{split}
\end{equation}

Since $U(\mathcal{L})$ is generated by $\mathcal{O}_X$ and the enveloping algebra of $\mathcal{L}$, subject to the relations $(fl-lf)=\rho(l)(f)$ and $fl=f.l$, it follows from the equations and definition that $U(\mathcal{L})$ is $G$-equivariant as $\mathcal{O}_X$-module and furthermore, axioms $i)$ and $ii)$ of Definition \ref{d1defequivdiffalg} are satisfied. Further it easy to check that $G$-action preserves the filtration on $U(\mathcal{L})$.

The morphism $i_{\mathfrak{g}}:r\mathfrak{g} \to \mathcal{L}$ can be extended to a morphism $i_{\mathfrak{g}}:r\mathfrak{g} \to U(\mathcal{L})$ via the natural map $\mathcal{L} \to U(\mathcal{L})$ and it is clear by construction that under $G$-action and the map $i_{\mathfrak{g}}$ defined above that $U(\mathcal{L})$ becomes an $r$-deformed $G$-equivariant differential algebra.
\end{proof}

As we are interested in deformed Picard algebroids, we define the notion of an $r$-deformed $G$-equivariant Picard algebroid. The $G$ action on $\mathcal{L}$ induces by differentiation a $\mathfrak{g}:=\Lie(G)$ action via  a map $\beta_{\mathcal{M}}:\mathfrak{g} \to \End(\mathcal{L})$. We also let $\eta: \mathfrak{g} \to \mathcal{T}_X$ denote the infinitesimal action of $\mathfrak{g}$ on $X$.

\begin{definition}
\label{d1definitionrdeformedPicard}
Let $\mathcal{L}$ be an $r$-deformed Picard algebroid. We say that $\mathcal{L}$ is an $r$-deformed $G$-equivariant Picard algebroid if $\mathcal{L}$ is an $r$-deformed $G$-equivariant algebroid, in the short exact sequence

$$0 \to \mathcal{O}_X \to \mathcal{L} \to r\mathcal{T}_X \to 0$$

all the morphisms are $G$-equivariant and 

\begin{itemize}
\item{The derivative of the $G$-action induces a $\mathfrak{g}$ action and thus a $r \mathfrak{g}$ action on $\mathcal{L}$. This must coincide with the action $l \mapsto [i_{\mathfrak{g}}(\psi),l]$ for $\psi \in r\mathfrak{g}$ and $ l \in \mathcal{L}$.}
\item{$\eta_{|r \mathfrak{g}}=\rho \circ i_{\mathfrak{g}}$.}
\end{itemize}

A morphism of $r$-deformed $G$-equivariant Picard algebroids is a morphism of $r$-deformed Picard algebroids compatible with the equivariance structures.
\end{definition}

\begin{definition}
\label{d1definitiondeformedhtdo}
Let $\mathcal{D}$ be a $\mathcal{O}_X$-algebra. We say that $\mathcal{D}$ is a sheaf of $r$-deformed $G$-homogeneous twisted differential operators ($r$-deformed $G$-htdo) if $(\mathcal{D},i_{\mathfrak{g}})$ is an $r$-deformed $G$-equivariant differential $\mathcal{O}_X$-algebra and a sheaf of $r$-deformed twisted differential operators, and furthermore: 

\begin{itemize}
\item{The image of $i_{\mathfrak{g}}$ lies in $F_1 \mathcal{D}$.}
\item{The derivative of the $G$-action induces a $\mathfrak{g}$ action and thus a $r \mathfrak{g}$-action on $\mathcal{D}$. This must coincide with the action $d \mapsto [i_{\mathfrak{g}}(\psi),d]$ for $\psi \in r\mathfrak{g}$ and $ d \in \mathcal{D}$.}
\item{ $\eta_{|r \mathfrak{g}}=\rho \circ i_{\mathfrak{g}}$, where $\rho:F_1 \mathcal{D}=\Lie(\mathcal{D}) \to \mathcal{T}_X$ is the natural anchor map.}
\end{itemize}

A morphism of $r$-deformed $G$-htdo's is a morphism of $r$-deformed tdo's compatible with the equivariance structures.

\end{definition}

One should notice that since $g.1=1$, the morphism $\mathcal{O}_X \to F_0(\mathcal{D})$ is automatically $G$-equivariant.

\begin{lemma}
\label{d1eqPicalgtohtdo}
Let $(\mathcal{L},\rho,i_{\mathfrak{g}})$ be an $r$-deformed $G$-equivariant Picard algebroid. Then $\mathscr{T}(\mathcal{L})$ is an $r$-deformed $G$-htdo.
\end{lemma}

\begin{proof}

We have by Lemma \ref{d1equivariantdifalg} that $U(\mathcal{L})$ is a $G$-equivariant differential algebra. Now, since $\mathcal{L}$ is $G$-equivariant, the action of $G$ stabilises the ideal generated by $i(1)-j(1)$, so the $G$ action descends on $U(\mathcal{L})/U(\mathcal{L})(i(1)-j(1))$. Similarly, composing the map $i_{\mathfrak{g}}: r\mathfrak{g} \to U(\mathcal{L})$ with the natural projection, we obtain a map $r\mathfrak{g} \to \mathscr{T}(\mathcal{L}).$ Further, by Proposition \ref{d1liealgebroidtotdo}, $U(\mathcal{L})/U(\mathcal{L})(i(1)-j(1))$ is an $r$-deformed tdo. Finally, the axioms of Definition \ref{d1definitionrdeformedPicard} imply that $\mathscr{T}(\mathcal{L})$ is an $r$-deformed $G$-htdo.
\end{proof} 

\begin{lemma}
\label{d1htdotoeqPicalg}
Let $\mathcal{D}$ be an $r$-deformed $G$-htdo. Then $\mathcal{L}:=F_1(\mathcal{D})$ is an $r$-deformed $G$-equivariant Picard algebroid.

\end{lemma}

\begin{proof}

We have by Lemma \ref{d1tdotoliealgebroid} that $\mathcal{L}$ is an $r$-deformed Picard algebroid and the axioms for $\mathcal{D}$ imply that $\mathcal{L}$ is a $G$-equivariant. Since $\mathcal{D}$ is a $G$-equivariant differential algebra the map $\mathcal{L} \to \mathcal{T}_X$ is $G$-equivariant. Lastly, the morphism $\mathcal{O}_X \to F_0(\mathcal{D})$ is $G$-equivariant and the other axioms in definition Definition \ref{d1definitionrdeformedPicard} follow from the corresponding axioms in \ref{d1definitiondeformedhtdo}.
\end{proof}

\begin{corollary}
\label{d1htdoPicalgcorresp}
Let $X$ be an $R$-variety and $r \in R$ a regular element. The maps $\mathscr{T}$ and $\Lie$ induce quasi-inverse equivalences of categories between the category of $r$-deformed $G$-equivariant Picard algebroids on $X$ and the category of $r$-deformed $G$-htdo's on $X$.

\end{corollary}

\begin{proof}
This follows by combining Lemmas  \ref{d1eqPicalgtohtdo} and \ref{d1htdotoeqPicalg} and Corollary \ref{d1Picardalgtdocorresp}.
\end{proof}

\section{Pullback of deformed Picard algebroids}
\label{d1sectionpullbackofLiealg}

Throughout this section, we fix $f:Y \to X$ a morphism of $R$-varieties and $r \in R$ a regular element. The map $f$ induces a morphism $f^* \Omega_{X}^{1} \to \Omega_{Y}^{1}$ and by dualising we obtain $\alpha:\mathcal{T}_Y \to f^*{\mathcal{T}_X}$. Here $\Omega_{X}$ and $\Omega_{Y}$ denote the sheaf of differential $1$-forms.

\begin{definition}
Let $(\mathcal{L},\rho_X)$ be an $r$-deformed Picard algebroid on $X$ and let $\beta=f^*(\rho_X)$. Then we let $$f^{\#} \mathcal{L}:=r\mathcal{T}_Y \times_{rf^* \mathcal{T}_X} f^*\mathcal{L}=\{(d,l)| d \in r\mathcal{T}_Y, l \in f^*\mathcal{L},\alpha(d)=\beta(l)\}.$$

We give $f^{\#} \mathcal{L}$ the structure of a Lie algebroid by setting $\rho_Y(d,l)=d$ and the Lie bracket be induced by

$$[(\psi,f \otimes P),(\eta, g\otimes Q)]:=([\psi,\eta], fg \otimes [P,Q]+\psi(g) \otimes P-\eta(f) \otimes Q),$$

 for  $\psi,\eta \in r\mathcal{T}_Y,f,g \in \mathcal{O}_Y,P,Q \in f^{-1} \mathcal{L}.$ We call $f^{\#}\mathcal{L}$ the pullback of $\mathcal{L}$.
\end{definition}


\begin{lemma}
\label{d1algebroidnoneqpullback}
Let $(\mathcal{L},\rho_X)$ be an $r$-deformed Picard algebroid. Then $(f^{\#} \mathcal{L},\rho_Y)$ is an $r$-deformed Picard algebroid. 

\end{lemma}

\begin{proof}

Since $\mathcal{L}$ is an $r$-deformed Picard algebroid, it fits into a short exact sequence:
  $$ 0 \to \mathcal{O}_X \to \mathcal{L} \to r\mathcal{T}_X \to 0.$$

By assumption, $X$ is a smooth variety, thus $\mathcal{T}_X$ is a free $\mathcal{O}_X$-module, so in particular is flat. Since $r \mathcal{T}_X \cong \mathcal{T}_X$ as $\mathcal{O}_X$-modules, by pulling back along $f$ we obtain a short exact sequence:

\begin{equation}
\label{d1equationforprovingpullbackofalg}
  0 \to \mathcal{O}_Y \to f^*\mathcal{L} \to rf^* \mathcal{T}_X \to 0.  
\end{equation}

Considering the pullback  diagram

\begin{center}
\begin{tikzcd}
  &f^{\#} \mathcal{L} \arrow[d] \arrow[r]  & f^* \mathcal{L} \arrow[d]\\
  & r \mathcal{T}_Y \arrow[r]              &r f^* \mathcal{T}_X,
\end{tikzcd}    
\end{center}

 we obtain $\ker (\rho_Y)= \ker( f^{\#} \mathcal{L} \to \mathcal{T}_Y) \cong \ker(f^* \mathcal{L} \to rf^* \mathcal{T}_X) \cong \mathcal{O}_Y.$ Finally, since $\im(f^*\mathcal{L} \to rf^* \mathcal{T}_X)=rf^* \mathcal{T}_X$, we obtain that $\im(\rho_Y)=r\mathcal{T}_Y$.
\end{proof}

We would like to describe how the pullback interacts with composition of morphisms. In general, one would like to prove that for $u:Z\to Y, f:Y \to X$ maps of $R$-varieties and $\mathcal{L}$ an $r$-deformed Picard algebroid on $X$,  we have $u^{\#} f^{\#} \mathcal{L} \cong (f \circ u)^{\#} \mathcal{L}$. Unfortunately, this is not always true. In the following, we give sufficient conditions.

\begin{lemma}
\label{d1Liealgcompositionofpullback}
Let $u:Z \to Y, f:Y \to X$ be maps of $R$-varieties and $\mathcal{L}$ be an $r$-deformed Picard algebroid on $X$. Assume that $u$ is flat or $f$ is etale. Then

                $$u^{\#}f^{\#} \mathcal{L} \cong (f \circ u)^{\#} \mathcal{L}.$$
\end{lemma}

\begin{proof}
First assume $u$ is flat. Then:

\begin{equation}
\begin{split}
u^{\#} f^{\#} \mathcal{L}&=u^{\#}(r\mathcal{T}_Y \times_{rf^*\mathcal{T}_X} f^*\mathcal{L}) \\ 
                         &=r\mathcal{T}_Z \times_{ru^*\mathcal{T}_Y} u^*(r\mathcal{T}_Y \times_{rf^*\mathcal{T}_X} f^*\mathcal{L}) \\
                         &=r\mathcal{T}_Z \times_{ru^*\mathcal{T}_Y} {ru^*\mathcal{T}_Y} \times_{ru^*f^* \mathcal{T}_X} u^*f^* \mathcal{L} \text{ ($u$ is flat, so commutes with limits)} \\
                         &\cong r\mathcal{T}_Z  \times_{ru^*f^* \mathcal{T}_X} u^*f^* \mathcal{L}     \\
                         & \cong (f \circ u)^{\#} \mathcal{L}.
\end{split}
\end{equation}

Now suppose $f$ is etale, thus $f^*\mathcal{T}_X \cong \mathcal{T}_Y$. Then:

\begin{equation}
\begin{split}
u^{\#} f^{\#} \mathcal{L}&=u^{\#}(r\mathcal{T}_Y \times_{rf^*\mathcal{T}_X} f^*\mathcal{L}) \\
                         &\cong u^{\#} f^* \mathcal{L} \\
                         &\cong r\mathcal{T}_Z \times_{ru^* \mathcal{T}_Y} u^*f^* \mathcal{L} \\
                         &\cong r\mathcal{T}_Z \times_{ru^* f^* \mathcal{T}_X} u^*f^* \mathcal{L} \\
                         &\cong (f \circ u)^{\#} \mathcal{L}. \qedhere
\end{split}
\end{equation}
\end{proof}

\begin{lemma}
\label{d1algrebroideqpullback}
Let $(\mathcal{L}, \rho_X,i_{\mathfrak{g}})$  be an $r$-deformed $G$-equivariant Picard algebroid on $X$ and assume further $f:Y \to X$ is a $G$-equivariant. Then $ (f^{\#} \mathcal{L}, \rho_Y)$ is an $r$-deformed $G$-equivariant Picard algebroid.

\end{lemma}

\begin{proof}
By the Lemma above, $f^{\#} \mathcal{L} $ is an $r$-deformed Picard algebroid. Since $f:Y \to X$ is $G$-equivariant, we obtain by Lemma \ref{d1Oequivpreservefunctor} that  $f^*\mathcal{L}$ is a $G$-equivariant $\mathcal{O}_Y$-module and that the maps $\alpha$ and $\beta$ are $G$-equivariant. We define a $G$ action on $f^{\#} \mathcal{L}$ via

   $$g.(d,l)=(g.d, g.l) \text{ for } d \in r\mathcal{T}_Y, l \in f^* \mathcal{L}. $$

We need to check that this action is well-defined and $\rho_Y$ is $G$-equivariant. The second statement is easy, we have $g.\rho_Y(a,b)=\rho_Y(g.(a,b))$. For the other statement, let $d \in r\mathcal{T}_Y, l \in f^*\mathcal{L}$ such that $\alpha(d)=\beta(l)$. Then, we have

$$\alpha(g.d)=g.\alpha(d)=g.\beta(l)=\beta(g.l).$$

So, we are left to prove that $G$-action interacts correctly with the Lie bracket. Let  $\psi,\eta \in r\mathcal{T}_Y,f,h \in \mathcal{O}_Y,P,Q \in f^{-1} \mathcal{L}.$  Then

\begin{equation}
\begin{split}
g.[(\psi,f \otimes P),(\eta, h \otimes Q)] &=g.([\psi,\eta], fh \otimes [P,Q]+\psi(h) \otimes P-\eta(f) \otimes Q)\\
                                                           &=(g.[\psi,\eta],g.fh \otimes g.[P,Q]\\
                                                           &+g.\psi(h) \otimes g.P -g.\eta(f) \otimes g.Q) \\
                                                           &=([g.\psi,g.\eta],(g.f)(g.h) \otimes  [g.P,g.Q]\\
                                                           &+ \psi(g.h) \otimes g.P-\eta(g.f) \otimes g.Q) \\
                                                           &=[g.(\psi, f \otimes P),g.(\eta, h \otimes Q)].
\end{split}
\end{equation}

Since $G$ acts on $Y$ we obtain the infinitesimal map $\eta:\mathfrak{g} \to \mathcal{T}_Y$. The map $i_{\mathfrak{g}}: r \mathfrak{g} \to \mathcal{L}$ can be extended to a map $i_{\mathfrak{g}}: \mathcal{O}_X \otimes r \mathfrak{g} \to \mathcal{L}$ and by pulling back we obtain a map $i_{\mathfrak{g}}^* :\mathcal{O}_Y \otimes r \mathfrak{g} \to f^*\mathcal{L}$. We let $i: r\mathfrak{g} \to f^*\mathcal{L}$ be the restriction of $i_{\mathfrak{g}}^*$ to $r \mathfrak{g}$; by construction, we have that the $r \mathfrak{g}$ action induced by $i$ coincides with the one induced by the $G$ action. Therefore we obtain a map $(\eta_{|r \mathfrak{g}},i):r \mathfrak{g} \to f^{\#} \mathcal{L}$ and it follows by the construction that the $r\mathfrak{g}$-action induced by the derivative of the $G$-action coincides with the $r\mathfrak{g}$-action induced by $(\eta_{| r \mathfrak{g}},i)$ and further that $ \eta_{| r \mathfrak{g}} =\rho_Y \circ (\eta_{| r \mathfrak{g}},i)$. Thus $(f^{\#} \mathcal{L}, \rho_Y,(\eta_{| r \mathfrak{g}},i))$ is an $r$-deformed $G$-equivariant Picard algebroid. 
\end{proof}

\begin{definition}
Let $\mathcal{D}$ be an $r$-deformed tdo/$G$-htdo on $X$ and let $f:Y \to X$ be a morphism/$G$-equivariant morphism of $R$-varieties. We call

                                $$f^{\#} \mathcal{D} := \mathscr{T}(f^{\#} \Lie(\mathcal{D}))$$

the pullback of $\mathcal{D}$ along $f$.

\end{definition}

\begin{corollary}
\label{d1cc}
Let the notation as above. Then $f^{\#} \mathcal{D}$ is well defined and furthermore it is an $r$-deformed tdo/htdo.

\end{corollary}

\begin{proof}
This follows from Lemmas \ref{d1algebroidnoneqpullback}, \ref{d1algrebroideqpullback} and Corollaries \ref{d1Picardalgtdocorresp}, \ref{d1htdoPicalgcorresp}.
\end{proof}

\begin{corollary}
\label{d1tdocompositionofpullback}
Let $u:Z \to Y, f:Y \to X$ be maps of $R$-varieties and $\mathcal{D}$ be an $r$-deformed tdo on $X$. Assume that $u$ is flat or $f$ is etale. Then

                $$u^{\#}f^{\#} \mathcal{D} \cong (f \circ u)^{\#} \mathcal{D}.$$
\end{corollary}

\begin{proof}
This follows from Lemma \ref{d1Liealgcompositionofpullback} and Corollary \ref{d1Picardalgtdocorresp}.
\end{proof}



                        $$(e \times 1_X)^{\#} \alpha: \mathcal{D} \to \mathcal{D}$$

Let us explain how our condition for an $r$-deformed $G$-htdo fits into a diagram satisfying the cocycle condition. Denote the $G$-action by $\sigma_X: G \times X \to X$. Furthermore, we denote $p_X:G \times X \to X$ and $p_{2X}:G \times G \times X \to X$ the projections on the $X$ factor, $p_{23X}:G \times G \times X \to G \times X$ the projection onto the second and third factor and $m:G \times G \to G$ the multiplication of the group $G$.

\begin{lemma}

Let $\mathcal{D}$ be an $r$-deformed $G$-htdo on $X$. Then there exists $\alpha:\sigma_X^{\#} \mathcal{D} \to p_X^{\#} \mathcal{D}$ an isomorphism of $\mathcal{O}_{G \times X}$-algebras such that the diagram:

\begin{equation}
\label{Ghtdocommdia}
\begin{tikzcd}
&(1_G \times \sigma_X)^{\#}p_X^{\#}\mathcal{D} \arrow[r," p_{23X}^{\#} \alpha"]   &p_{2X}^{\#}\mathcal{D} \\
&(1_G \times \sigma_X)^{\#} \sigma_X^{\#} \mathcal{D} \arrow [u,"(1_G \times \sigma_X)^{\#} \alpha "] \arrow[r,leftrightarrow,"id "] &(m \times 1_X)^{\#}\sigma_X^{\#} \mathcal{D} \arrow[u, " (m \times 1_X)^{\#} \alpha"]
\end{tikzcd}
\end{equation}

of $\mathcal{O}_{G \times G \times X}$-algebras commutes ( the cocycle condition) and the pullback 
                        $$(e \times 1_X)^{\#} \alpha: \mathcal{D} \to \mathcal{D}$$
is the identity map. We note that this is the same condition as in \cite[Section 5.2]{Wan}. 
\end{lemma}

\begin{proof}

Let $\mathcal{L}=\Lie(D)$; we have by Lemma \ref{d1htdotoeqPicalg} that $\mathcal{L}$ is a $G$-equivariant $r$-deformed Picard algebroid. In particular, we obtain an isomorphism $\beta: \sigma_X^* \mathcal{L} \to p_X^* \mathcal{L}$, which can be extended to an isomorphism of $r$-deformed Picard algebroids $\sigma_X^{\#} \mathcal{L} \to p_X^{\#} \mathcal{L} $ and thus to an isomorphsim of $\mathcal{O}_{G \times X}$-algebras $\alpha: \sigma_X^{\#} \mathcal{D}=\mathscr{T}(\sigma_X^{\#} \mathcal{L}) \to \mathscr{T}(p_X^{\#} \mathcal{L})=p_X^{\#} \mathcal{D}$. Further, since the map $\beta$ satisfies the cocycle condition, so does $\alpha$.
\end{proof}

\section{Representations of Lie algebroids}
\label{d1Liealgrep}

Throughout this section we fix $X$ an $R$-variety, $r \in R$ a regular element and $(\mathcal{L},\rho_X)$ a Lie algebroid on $X$.

\begin{definition}
Let $\mathcal{M}$ be a quasi-coherent $\mathcal{O}_X$-module. We say that $\mathcal{M}$ is a \emph{$\mathcal{L}$-module} if $\mathcal{M}$ is a sheaf of modules over the sheaf of Lie algebras $\mathcal{L}$ and for all $f \in \mathcal{O}_X$, $l \in \mathcal{L}$, $m \in \mathcal{M}$, we have

\begin{equation}
\begin{split}
f.(l.m) &=l.(f.m)-\rho_X(l)(f).m, \\
(f.l).m &=f.(l.m) .
\end{split}
\end{equation}
\end{definition}

We define a morphism of $\mathcal{L}$-modules to be a morphism of $\mathcal{O}_X$-modules compatible with the $\mathcal{L}$-action.

\begin{definition}

Assume that $(\mathcal{L},\rho_X)$ is an $r$-deformed Picard algebroid and let $f:Y \to X$ be a  map of $R$-varieties and $\mathcal{M}$ a $\mathcal{L}$-module. Then we define the pullback of $\mathcal{M}$ along $f$, via $f^{\#} \mathcal{M}=f^*\mathcal{M}$ as an $\mathcal{O}_Y$-module and 

 $$(\psi, P \otimes l).( Q \otimes m):=\psi(Q) \otimes m +PQ \otimes l.m,$$ 

for $\psi \in r\mathcal{T}_Y$, $P,Q \in \mathcal{O}_Y, l \in f^{-1} \mathcal{L}, m \in f^{-1} \mathcal{M}$.
\end{definition}

\begin{lemma}
\label{d1repliealgpullbackwelldef}
The action defined above makes $f^{\#} \mathcal{M}$ a $f^{\#} \mathcal{L}$-module.

\end{lemma}

\begin{proof}
 
First, we check the bracket action. We have for $\psi,\eta \in r \mathcal{T}_Y$, $P,Q,R \in \mathcal{O}_Y$, $a,b \in f^{-1} \mathcal{L}$ and $m \in f^{-1} \mathcal{M}$ that

\begin{equation}
\begin{split}
(\psi, P \otimes a).((\eta, R \otimes b). Q \otimes m) &= (\psi, P \otimes a).(\eta(Q) \otimes m + RQ \otimes b.m) \\
                                                                             &=\psi(\eta(Q)) \otimes m+P \eta(Q) \otimes a.m \\
                                                                             &+ \psi(RQ) \otimes b.m+ PQR \otimes a.(b.m)
\end{split}
\end{equation}

and 

\begin{equation}
\begin{split}
(\eta, R \otimes b).((\psi, P \otimes a). Q \otimes m)&=\eta(\psi(Q)) \otimes m + R \psi(Q) \otimes b.m \\
                   &+\eta(PQ) \otimes a.m + PQR \otimes b.(a.m).
\end{split}
\end{equation}

Thus, combining the equations above, we obtain

\begin{equation}
\begin{split}
(\psi, P \otimes a).((\eta, R \otimes b). Q \otimes m) &-(\eta, R \otimes b).((\psi, P \otimes a). Q \otimes m) \\
                                                                             &=\psi(\eta(Q))-\eta(\psi(Q)) \otimes m\\
                                                                             &+ PQR \otimes a.(b.m)-b.(a.m) \\
                                                                             &+ P\eta(Q) -\eta(PQ) \otimes a.m+ \psi(RQ)-R \psi(Q) \otimes b.m \\
                                                                             &=[\psi,\eta](Q) \otimes m + PQR \otimes [a,b].m \\
                                                                             &-Q \eta(P) \otimes a.m + Q \psi(R) \otimes b.m\\
                                                                             &=[(\psi, P \otimes a),(\eta, R \otimes b)]. (Q \otimes m).
\end{split}
\end{equation}

To check the first axiom, we have

\begin{equation}
\begin{split}
R.((\psi, P \otimes l). Q \otimes m) &=R.(\psi(Q) \otimes m + PQ \otimes l.m) \\
                                                   &=R \psi(Q) \otimes m + PQR \otimes l.m \\
                                                   &= \psi(RQ) \otimes m - Q \psi(R) \otimes m + PQR \otimes l.m \\
                                                   &= (\psi, P \otimes l). R.( Q \otimes m) - \rho_Y(\psi, P \otimes l)(R).(Q \otimes m).
\end{split}
\end{equation}

Finally, we have

\begin{equation}
\begin{split}
R.((\psi, P \otimes l). Q \otimes m) &=R.(\psi(Q) \otimes m + PQ \otimes l.m) \\
                                                   &=R \psi(Q) \otimes m + PQR \otimes l.m \\
                                                   &=(R.(\psi, P \otimes l)). Q \otimes m. \qedhere
\end{split}
\end{equation}
\end{proof}

We now define equivariant representations. Let $G$ be a smooth affine algebraic group of finite type acting on an $R$-variety $X$.

\begin{definition}
\label{d1defliealgebroidequivrep}
Let $(\mathcal{L}, \rho,i_{\mathfrak{g}})$ be an $r$-deformed $G$-equivariant Lie algebroid on $X$. We say that $\mathcal{M}$ is a $G$-equivariant $\mathcal{L}$-module if:

\begin{enumerate}[label=\roman*)]
\item{$\mathcal{M}$ is a $G$-equivariant $\mathcal{O}_X$-module and $\mathcal{M}$ is a $\mathcal{L}$-module.}
\item{$g.(l.m)=(g.l).(g.m)$ for any  $g \in G$, $l \in \mathcal{L}$ and $m \in \mathcal{M}$.}
\item{ The $r\mathfrak{g}$ action induced by restricting the $\mathfrak{g}$ action induced from the  derivative of the $G$-action on $\mathcal{M}$ coincides with the $r\mathfrak{g}$-action induced from $i_{\mathfrak{g}}$.}
\end{enumerate}

A morphism of $G$-equivariant $\mathcal{L}$-modules is a morphism of $G$-equivariant $\mathcal{O}_X$-modules compatible with the $\mathcal{L}$-action.

\end{definition}

\begin{lemma}
\label{d1pullbackofequivLiealgrep}
Let $\mathcal{L}$ be an $r$-deformed $G$-equivariant Lie algebroid on $X$ and $\mathcal{M}$ a $G$-equivariant $\mathcal{L}$-module.  Further, let $Y$ be a variety and $f:Y \to X$ a  $G$-equivariant morphism. Then $f^{\#} \mathcal{M}$ is a $G$-equivariant $f^{\#} \mathcal{L}$-module.

\end{lemma}

\begin{proof}
We have the $G$ action on $f^{\#} \mathcal{M}$ induced by the action on the simple tensors $g.(Q \otimes m)=g.Q \otimes g.m$ for $g \in G$, $Q \in \mathcal{O}_Y$, $m \in f^{-1}\mathcal{M}$.  Since $f^{\#} \mathcal{M}= f^* \mathcal{M}$ as a $\mathcal{O}_Y$-module, the first axiom of Definition \ref{d1defliealgebroidequivrep} follows from Lemma \ref{d1Oequivpreservefunctor}.

Next, we have for $g \in G$, $\psi \in \mathcal{T}_Y, Q \in \mathcal{O}_Y, l \in f^{-1} \mathcal{L}, m \in f^{-1} \mathcal{M}$:

\begin{equation}
\begin{split}
g.((\psi, P \otimes l). Q \otimes m)&=g.(\psi(Q) \otimes m + PQ \otimes l.m) \\
                                                  &=g. \psi(Q) \otimes g.m +g.(PQ) \otimes g.(l.m) \\
                                                  &=(g. \psi) (g.Q) \otimes g.m + g.P g.Q \otimes (g.l).(g.m) \\
                                                  &=(g.\psi, g.(P \otimes l) . (g.Q \otimes g.m) \\
                                                  &=(g.(\psi, P \otimes l)).(g.(Q \otimes m)).
\end{split}
\end{equation}

Finally, the third axiom follows easily from the definition of $G$-action on $f^{\#} \mathcal{L}$ and on $f^{\#} \mathcal{M}$.
\end{proof}

We should remark that one could define a more general notion of equivariance over an $r$-deformed $G$-equivariant Lie algebroid $(\mathcal{L},\rho,i_g)$. Let $L$ be a closed subgroup of $G$, then one may relax condition $i)$ to $L$-equivariance, impose condition $ii)$ to hold just for $g \in L$ and change condition $iii)$ to: the $r\mathfrak{l}=\Lie(L)$ action induced by derivative of the $L$-action on $\mathcal{M}$ coincides with the $r\mathfrak{l}$-action induced by the restriction of $i_{\mathfrak{g}}$ to $r\mathfrak{l}$.

\section{Modules over twisted differential operators}
\label{d1sectiontdomod}

We keep the notation from the previous section. As our main interest is in modules over deformed htdo's, we need a definition of a representation of a deformed Picard algebroid. Recall that $(\mathcal{L},\rho)$ is an $r$-deformed Picard algebroid if  $\mathcal{L}$ fits into the following short exact sequence $0 \to \mathcal{O}_X \to \mathcal{L} \to r\mathcal{T}_X \to 0$, Thus, for a $\mathcal{L}$-module, we get two actions of the structure sheaf: one since $\mathcal{M}$ is an $\mathcal{O}_X$-module by assumption and one induced by the short exact sequence. 

We say that $\mathcal{M}$ is a \emph{Picard module} if the two actions defined above coincide.

\begin{lemma}

Let $(\mathcal{L},\rho)$ be an $r$-deformed Picard algebroid and let $\mathcal{M}$ be a Picard $\mathcal{L}$-module. Then $\mathcal{M}$ is a module over $\mathcal{D}:=\mathscr{T}(\mathcal{L})$.
\end{lemma}

\begin{proof}
Since $\mathcal{M}$ is a $\mathcal{L}$-module, we obtain that $\mathcal{M}$ is also a $U(\mathcal{L})$-module in a similar fashion as Lie algebra representations correspond to enveloping algebra modules. Further, the condition that $\mathcal{M}$ is a Picard is exactly the condition that allows us to descend to a $\mathcal{D}$-module.
\end{proof}

As a corollary of the proof, we obtain immediately:

\begin{corollary}
\label{d1repequivalencenonequivliealgtdo}
Let $(\mathcal{L},\rho)$ be an $r$-deformed Picard algebroid on $X$ and $\mathcal{D}=\mathscr{T}(\mathcal{L})$. Then the identity map provides a one-to-one correspondence between Picard $\mathcal{L}$-modules and $\mathcal{D}$-modules.

\end{corollary}

For an $r$-deformed Picard algebroid $\mathcal{L}$ we will denote $\Mod(\mathcal{L})$ the category of Picard $\mathcal{L}$-modules. Similarly, for an $r$-deformed tdo $\mathcal{D}$ we denote $\Mod(\mathcal{D})$ the category of quasi-coherent $\mathcal{D}$-modules and $\Coh(\mathcal{D})$ its full subcategory consisting of coherent modules.

\textbf{Equivariant representations}

We now move to the equivariant setting. Recall that $G$ is an algebraic group acting on $X$ with Lie algebra $\mathfrak{g}$.

\begin{definition}
\label{d1definitionofGequivariantmodule}
Let $(\mathcal{D},i_{\mathfrak{g}})$ be an $r$-deformed $G$-htdo. We say that a $\mathcal{D}$-module $\mathcal{M}$ is a $G$-equivariant module over $\mathcal{D}$ if:

\begin{enumerate}[label=\roman*)]
\item{$\mathcal{M}$ is a $G$-equivariant  $\mathcal{O}_X$-module.}
\item{$g.(d.m)=(g.d).(g.m)$, for all $g \in G, d \in \mathcal{D}, m \in \mathcal{M}$.}
\item{ The $r\mathfrak{g}$-action induced by the derivative of the $G$-action on $\mathcal{M}$ coincides with the $r\mathfrak{g}$-action induced by $i_{\mathfrak{g}}$.}
\end{enumerate}

\end{definition}

\begin{lemma}

Let $(\mathcal{L},\rho,i_{\mathfrak{g}})$ be an $r$-deformed $G$-equivariant Picard algebroid and $\mathcal{M}$ be a $G$-equivariant Picard $\mathcal{L}$-module. Further, let $\mathcal{D}=\mathscr{T}(\mathcal{L})$ the $r$-deformed $G$-htdo corresponding to $\mathcal{L}$. Then $\mathcal{M}$ is a $G$-equivariant $\mathcal{D}$-module.

\end{lemma}

\begin{proof}

We have by Corollary \ref{d1repequivalencenonequivliealgtdo} that $\mathcal{M}$ is a $\mathcal{D}$-module, so we only have to prove equivariance.  Axioms $i)$ and $iii)$ follow from the corresponding axioms in Definition \ref{d1defliealgebroidequivrep}, while axiom $ii)$ follows by an easy induction argument by using the definition of the $G$-action on $U(\mathcal{L})$. We only check the first step: we have that for $g \in G$, $l_1,l_2 \in \mathcal{L}$ and $m \in \mathcal{M}$ that

\begin{equation}
\begin{split}
g.(l_1l_2 .m)&=g.(l_1.(l_2.m) \\
                  &=(g.l_1).(g.(l_2.m) \\
                  &=[(g.l_1)(g.l_2)].(g.m)\\
                  &=(g.l_1l_2).(g.m). \qedhere
\end{split}
\end{equation}
\end{proof}

Similarly to the non-equivariant case, we obtain

\begin{corollary}
\label{d1repequivalentequivliealgtdo}

Let $X$ be an $R$-variety and $(\mathcal{L},\rho,i_{\mathfrak{g}})$ an $r$-deformed $G$-equivariant Picard algebroid and let $\mathcal{D}=\mathscr{T}(\mathcal{L})$ be the corresponding $r$-deformed $G$-htdo. The identity map provides a one-to-one correspondence between $G$-equivariant Picard $\mathcal{L}$-modules and $G$-equivariant $\mathcal{D}$-modules.

\end{corollary}

For an $r$-deformed $G$-equivariant Picard algebroid $\mathcal{L}$ we will denote $\Mod(\mathcal{L},G)$ the category of $G$-equivariant Picard $\mathcal{L}$-modules. Similarly, for a $G$-htdo $\mathcal{D}$ we denote $\Mod(\mathcal{D},G)$ the category of quasi-coherent $G$-equivariant $\mathcal{D}$-modules and $\Coh(\mathcal{D},G)$ its full subcategory consisting of coherent modules. A similar argument to the one in Proposition \ref{d1GequivariantAbeliancat} proves that these categories are Abelian.

\textbf{Pullback of modules over twisted differential operators}

We conclude the section by defining of the pullback of a module over a sheaf of twisted differential operators. 

\begin{definition}
 
Let $f:Y \to X$ be a morphism of $R$-varieties, $\mathcal{D}$ an $r$-deformed tdo on $X$ and $\mathcal{M}$ a $\mathcal{D}$-module. We define the pullback  of $\mathcal{M}$ under the map $f$ to be $f^{\#} \mathcal{M}$.

\end{definition}

We should remark that the sheaf $f^{\#} \mathcal{M}$ has the structure of a $f^{\#} \mathcal{D}$-module by Corollary \ref{d1Picardalgtdocorresp}, Lemma \ref{d1repliealgpullbackwelldef} and Corollary \ref{d1repequivalencenonequivliealgtdo}.

\begin{lemma}

Let $\mathcal{D}$ be an $r$-deformed $G$-htdo, $\mathcal{M}$ a $G$-equivariant $\mathcal{D}$-module and $f:Y \to X$ a $G$-equivariant morphism of $R$-varieties. Then $f^{\#} \mathcal{M}$ is a $G$-equivariant $f^{\#} \mathcal{D}$-module.

\end{lemma}

\begin{proof}

We have by the definition above that $f^{\#} \mathcal{M}$ is a $f^{\#} \mathcal{D}$ module and by Corollary \ref{d1repequivalencenonequivliealgtdo}, $f^{\#} \mathcal{D}$ is an $r$-deformed $G$-htdo, so the statement makes sense. The claim now follows by combining Lemma \ref{d1pullbackofequivLiealgrep} and Corollary \ref{d1repequivalentequivliealgtdo}.
\end{proof}


As for the $r$-deformed $G$-htdo's, we may prove that a $G$-equivariant module over an $r$-deformed $G$-htdo module satisfies a cocycle condition. Denote the $G$-action on $X$ by $\sigma_X: G \times X \to X$. Further, we denote $p_X:G \times X \to X$ and $p_{2X}:G \times G \times X \to X$ the projections on the $X$ factor, $p_{23X}:G \times G \times X \to G \times X$ the projection onto the second and third factor and $m:G \times G \to G$ the multiplication of the group $G$. Let $\mathcal{D}$ be an $r$-deformed $G$-htdo and $\mathcal{M}$ a $G$-equivariant $\mathcal{D}$-module. Then there exists $\alpha:\sigma_X^{\#} \mathcal{M} \to p_X^{\#} \mathcal{M}$  an isomorphism of $p_X^{\#} \mathcal{D}$-modules such that the diagram \ref{Ghtdocommdia} commutes, where we replace $\mathcal{D}$ by $\mathcal{M}$. Again, we note that the condition is similar to the one in \cite[Section 5.2.9]{Wan}.

\section{Equivariant descent for Lie algebroids and homogeneous twisted differential operators}
\label{d1sectiondescentliealg}

Throughout this section, we let $X,Y$ be $R$-varieties, $G$ a smooth affine algebraic group of finite type acting \emph{freely} on $Y$. Further, we fix $f:Y \to X$ a locally trivial $G$-torsor and $(\mathcal{L},\rho,i_{\mathfrak{g}})$ an $r$-deformed $G$-equivariant Picard algebroid on $Y$. Let $\alpha: \mathfrak{g} \to \mathcal{T}_Y$ be the derivative of the $G$-action on $Y$. Recall that by Definition \ref{d1definitionrdeformedPicard}, $\alpha= \rho \circ i_{\mathfrak{g}}$ as maps from $r \mathfrak{g}$ to $\mathcal{T}_X$.  Throughout this section, we will use without further specifying that all the $R$-modules appearing have no $r$-torsion.

Let $\widetilde{\mathcal{T}}_X:=(f_* \mathcal{T}_Y)^G$ and let $\sigma: \widetilde{\mathcal{T}}_X \to \mathcal{T}_X$ denote the anchor map. Further, we denote $\widetilde{\mathfrak{g}_X}:=(f_* \mathcal{O}_Y \otimes \mathfrak{g})^G$, where $G$ acts on $\mathfrak{g}$ via the Adjoint action. We let $\tilde{\alpha}$ the induced map $\widetilde{\mathfrak{g}_X} \to \widetilde{\mathcal{T}}_X$. Since $f_* \mathcal{O}_Y \otimes \mathfrak{g}$ has no $r$-torsion, it is easy to see that $r \widetilde{\mathfrak{g}_X} \cong (f_* \mathcal{O}_Y \otimes r \mathfrak{g})^G$.

\begin{lemma}
\label{d1sesliealgeqdes}

The maps $\tilde{\alpha}$ and $\sigma$ induce a short exact sequence

$$ 0  \to \widetilde{\mathfrak{g}_X} \xrightarrow{\tilde{\alpha}}  \widetilde{\mathcal{T}}_X \xrightarrow{\sigma} \mathcal{T}_X \to 0.$$

\end{lemma}

\begin{proof}

The question is local, so we may assume that $X$ is affine $Y=G \times X$, $G$ acts on $Y$ via left multiplication on the first factor and $f$ is the projection on the second factor. In that case, we have 

\begin{equation}
\begin{split}
\widetilde{\mathfrak{g}_X}(X):&=(f_* \mathcal{O}_Y \otimes \mathfrak{g})^G(X) \\
                              &= (\mathcal{O}_Y(Y) \otimes \mathfrak{g})^G \\
                              &\cong (\mathcal{O}_X(X) \otimes \mathcal{O}_G(G) \otimes \mathfrak{g})^G \\
                              & \cong \mathcal{O}_X(X) \otimes ( \mathcal{O}_G(G) \otimes \mathfrak{g})^G \\
                              & \cong \mathcal{O}_X(X) \otimes (\mathcal{T}(G))^G. 
\end{split}             
\end{equation}

Further, we have by the proof of \cite[Lemma 4.4]{Annals} that $\widetilde{\mathcal{T}}_X(X) \cong \mathcal{O}_X(X) \otimes \mathcal{T}(G)^G  \oplus \mathcal{T}_X(X)$, so the conclusion follows.
\end{proof}

By abuse of notation we denote $i_{\mathfrak{g}}$ the map $i_{\mathfrak{g}}:\mathcal{O}_Y \otimes r\mathfrak{g} \to \mathcal{L}$ and $\tig:r\widetilde{\mathfrak{g}_X} \to (f_*\mathcal{L})^G$ the induced map. 

\begin{definition}

Let $f_{\#} \mathcal{L}^G$ be the $\mathcal{O}_X$-module $(f_* \mathcal{L})^G / \tig (r\widetilde{\mathfrak{g}_X})$. We call this the descent of $\mathcal{L}$.

\end{definition}

\begin{lemma}
\label{d1descenteqpicalg}
The $\mathcal{O}_X$-module $f_{\#} \mathcal{L}^G$ has the structure of an $r$-deformed Picard  algebroid.

\end{lemma}

\begin{proof}

The bracket structure on $\widetilde{\mathcal{L}}:=(f_* \mathcal{L})^G$ is induced from the bracket structure on $\mathcal{L}$;  furthermore this descends to a bracket structure on $f_{\#} \mathcal{L}^G$ by setting $[a+\tig (r\widetilde{\mathfrak{g}_X}),b+\tig (r\widetilde{\mathfrak{g}_X})]=[a,b]+\tig (r\widetilde{\mathfrak{g}_X})$ for $a,b \in (f_* \mathcal{L})^G$. Since the image of $i_{\mathfrak{g}}$ is an ideal in $\mathcal{L}$, we obtain that the image of $\tig$ is an ideal in $(f_* \mathcal{L})^G$, thus the bracket is well defined. Therefore, we are left to construct an anchor map.

Consider the short exact sequence $0 \to \mathcal{O}_Y \to \mathcal{L} \xrightarrow{\rho} r\mathcal{T}_Y \to 0$. Applying Proposition \ref{d1Omodulesequivariantdescent}, we obtain a short exact sequence

         $$ 0 \to \mathcal{O}_X \to \widetilde {\mathcal{L}} \xrightarrow{\tilde{\rho}} r\widetilde{\mathcal{T}}_X \to 0.$$

By construction, we have $\tilde{\alpha}= \tilde{\rho} \circ \tig$, so the map $$\bar{\widetilde{\rho}}: \widetilde{\mathcal{L}}/  \tig(r\widetilde{\mathfrak{g}_X}) \to \widetilde{\mathcal{T}}_X/ \tilde{\alpha} (r\widetilde{\mathfrak{g}_X})$$

is well defined. Furthermore, since there is no $r$-torsion, we have by Lemma \ref{d1sesliealgeqdes} an induced isomorphism $\bar{\sigma}: r\widetilde{\mathcal{T}}_X/ r\tilde{\alpha}(\widetilde{\mathfrak{g}_X}) \to r\mathcal{T}_X$. We define the anchor map on $\widetilde{\mathcal{L}}/ \tig (r\widetilde{\mathfrak{g}_X})$ to be $\rho_X:= \bar{\sigma} \circ \bar{\widetilde{\rho}}$ and it is clear that together with the bracket $\widetilde{\mathcal{L}}/ \tig (r\widetilde{\mathfrak{g}_X})$ becomes a Lie algebroid, so it remains to show that it is an $r$-deformed Picard algebroid. It is easy to see that $\im \rho_X=r \mathcal{T}_X$, so it is enough to prove that $\ker (\rho_X) \cong \ker \widetilde{\rho} \cong \mathcal{O}_X$.

Since $\bar{\sigma}$ is injective we have that $$\ker{\rho_X}= \ker(\bar{\tilde{\rho}})=(\ker \tilde{\rho}+ \tig(r\widetilde{\mathfrak{g}_X}))/\tig(r\widetilde{\mathfrak{g}_X}) \cong \ker \tilde{\rho}/ \ker \tilde{\rho} \cap \tig(r\widetilde{\mathfrak{g}_X})$$ 

by the second isomorphism theorem. Thus, it is enough to prove $\ker \tilde{\rho} \cap \tig(r\widetilde{\mathfrak{g}_X})= 0$ and since $\tilde{\alpha}=\tilde{\rho} \circ \tig$, this reduces to proving $\ker \tilde{\alpha}=0$. By assumptions, the action of $G$ on $Y$ is free; thus the map $\alpha: \mathfrak{g} \to \mathcal{T}_Y$ is injective and since $\mathcal{O}_Y$ is a faithfully flat $R$-module, the induced map $\alpha: \mathcal{O}_Y \otimes r\mathfrak{g} \to r \mathcal{T}_Y$ injective and thus so is $\tilde{\alpha}$.
\end{proof}
%

\begin{lemma}
\label{d1descentinverseseqdesliealg}

Let $\mathcal{L}$ as before. Then $f^{\#} (f_{\#} \mathcal{L}^G) \cong \mathcal{L}$.
\end{lemma}

\begin{proof}
Recall that by abuse of notation, we denote $i_{\mathfrak{g}}: \mathcal{O}_Y \otimes r\mathfrak{g} \to \mathcal{L}$ and $\alpha: \mathcal{O}_Y \otimes r\mathfrak{g} \to \mathcal{T}_Y$ the induced maps; we still have $\alpha=\rho \circ i_{\mathfrak{g}}$. Let $\widetilde{\mathcal{L}}:=(f_* \mathcal{L})^G$ so that there is a short exact sequence 

$$ 0 \to (f_* \mathcal{O}_Y \otimes r\mathfrak{g})^G \to \widetilde{\mathcal{L}} \to f_{\#} \mathcal{L}^G \to 0$$

Pulling back under the torsor $f$ and applying Proposition \ref{d1Omodulesequivariantdescent}, we obtain a short exact sequence

$$ 0 \to \mathcal{O}_Y \otimes r\mathfrak{g} \xrightarrow{i_{\mathfrak{g}}} \mathcal{L} \to f^* (f_{\#} \mathcal{L}^G) \to 0,$$

so, $f^* (f_{\#} \mathcal{L}^G) \cong \mathcal{L} /i_{\mathfrak{g}}(\mathcal{O}_Y \otimes r\mathfrak{g}).$

Similarly, by pulling back under $f$ the short exact sequence in Lemma \ref{d1sesliealgeqdes} and taking into account there is no $r$-torsion we obtain a short exact sequence

\begin{equation}
\label{d1usefulrandomses}
 0 \to \mathcal{O}_Y \otimes r\mathfrak{g} \xrightarrow{\alpha} r\mathcal{T}_Y \to rf^* \mathcal{T}_X \to 0,
\end{equation}
so $rf^* \mathcal{T}_X \cong r\mathcal{T}_Y/ \alpha(\mathcal{O}_Y \otimes r\mathfrak{g}).$

Therefore, we obtain

$$f^{\#} (f_{\#} \mathcal{L}^G) \cong f^* (f_{\#} \mathcal{L}^G) \times_{rf^* \mathcal{T}_X} r\mathcal{T}_Y \cong \mathcal{L}/i_{\mathfrak{g}}(\mathcal{O}_Y \otimes r\mathfrak{g})\times_{r \mathcal{T}_Y/\alpha(\mathcal{O}_Y \otimes r\mathfrak{g})} r\mathcal{T}_Y.$$

Define $\varphi:\mathcal{L} \to f^{\#} (f_{\#} \mathcal{L}^G)$ by $\varphi(l)=(l+i_{\mathfrak{g}}(\mathcal{O}_Y \otimes r\mathfrak{g}), \rho(l))$ for any $l \in \mathcal{L}$.

To see that $\varphi$ is injective, we use that $\alpha$ is an injective map, thus so is $i_{\mathfrak{g}}$, therefore $i_{\mathfrak{g}}(\mathcal{O}_Y \otimes r \mathfrak{g}) \cap \ker(\rho)= 0 $.

Finally, let $(l+i_{\mathfrak{g}}(\mathcal{O}_Y \otimes r\mathfrak{g}), x) \in \mathcal{L}/i_{\mathfrak{g}}(\mathcal{O}_Y \otimes r\mathfrak{g})\times_{ r\mathcal{T}_Y/\alpha(\mathcal{O}_Y \otimes r\mathfrak{g})} r\mathcal{T}_Y,$ so that $\rho(l)+\alpha(\mathcal{O}_Y \otimes r\mathfrak{g})= x+ \alpha( \mathcal{O}_Y \otimes r\mathfrak{g})$, thus $x=\rho(l)+i$ for some $i \in \alpha( \mathcal{O}_Y \otimes r\mathfrak{g})$. Since $\alpha$ is injective there exist a unique $j \in \mathcal{O}_Y \otimes r\mathfrak{g}$, so that $\alpha(j)=i$. Thus $\varphi(l+i_{\mathfrak{g}}(j))=(l+ i_{\mathfrak{g}}(\mathcal{O}_Y \otimes r \mathfrak{g}), x),$ so $\varphi$ is surjective.
\end{proof}

\begin{lemma}
\label{d1eqdescentPicardrightinverse}
Let $\mathcal{P}$ be an $r$-deformed Picard algebroid on $X$. Then $f_{\#}(f^{\#} \mathcal{P})^G \cong \mathcal{P}$.

\end{lemma}

\begin{proof}

We may view $\mathcal{P}$ as an $r$-deformed $G$-equivariant Picard algebroid by letting $g.p=p$ for all $g \in G, p \in P$ and the map $r\mathfrak{g} \to \mathcal{P}$ to be the $0$ map. Then by Lemma \ref{d1algrebroideqpullback}, $f^{\#} \mathcal{P}$ is an $r$-deformed $G$-equivariant Picard algebroid. The induced map $i_{\mathfrak{g}}: r\mathfrak{g} \to f^* \mathcal{P} \times_{rf^{*} \mathcal{T}_X} r\mathcal{T}_Y$ is defined as $i_{\mathfrak{g}}(\psi)=(0, \alpha(\psi))$.  

Using Proposition \ref{d1Omodulesequivariantdescent} for  $\mathcal{O}$-modules together with the fact that $(f_*)^G$ commutes with limits we get:

\begin{equation}
\begin{split}
(f_* f^{\#} \mathcal{P})^G &= (f_*(f^* \mathcal{P} \times_{rf^*\mathcal{T}_X} r\mathcal{T}_Y))^G \\
                        & \cong (f_*f^*\mathcal{P})^G \times_{r(f_*f^*\mathcal{T}_X)^G} (rf_* \mathcal{T}_Y)^G \\
                        & \cong \mathcal{P} \times_{r\mathcal{T}_X} r\widetilde{\mathcal{T}}_X.
\end{split}
\end{equation}

Therefore, we obtain:

\begin{equation}
\begin{split}
f_{\#} (f^{\#} \mathcal{P})^G &\cong (f_* f^{\#} \mathcal{P})^G/i_{\mathfrak{g}}(r\widetilde{\mathfrak{g}_X}) \\
                              & \cong (\mathcal{P} \times_{r\mathcal{T}_X} r\widetilde{\mathcal{T}}_X)/ i_{\mathfrak{g}}(r\widetilde{\mathfrak{g}_X}) \\
                             & \cong \mathcal{P} \times_{r\mathcal{T}_X} r\mathcal{T}_X \text{ (by Lemma \ref{d1sesliealgeqdes})}. \\
                            & \cong \mathcal{P}. \qedhere
\end{split}
\end{equation}
\end{proof}

\begin{corollary}
\label{d1noneqdescentpicalgcorr}
The functors $f_{\#}(-)^G$ and $f^{\#}(-)$ induce quasi-inverse equivalences of categories between the category of $r$-deformed $G$-equivariant Picard algebroids on $Y$ and the category of $r$-deformed Picard algebroids on $X$.

\end{corollary}

\begin{proof}
The proof follows from Lemmas \ref{d1descentinverseseqdesliealg} and \ref{d1eqdescentPicardrightinverse}.
\end{proof}

We consider again a more general setting: let $B$ another smooth affine algebraic group acting on $Y$ and $X$ such that $f:Y \to X$ is $B$-equivariant. Recall that we are assuming that $f$ is a locally trivial $G$-torsor.

%
%
%

\begin{lemma}
\label{d1Bequivariantdescentpicalg}
With the assumption as above we have:
\begin{itemize}
\item{let $(\mathcal{L},\rho,i_{\mathfrak{g} \times \mathfrak{b}})$ be an $r$-deformed $G \times B$-equivariant Picard algebroid on $Y$. Then $f_{\#}\mathcal{L}^G$ may be given the structure of an $r$-deformed $B$-equivariant Picard algebroid on X.}
\item{let $(\mathcal{K},\rho_X,j_{\mathfrak{b}})$ be an $r$-deformed $B$-equivariant Picard algebroid on $X$. Then $f^{\#} \mathcal{K}$ may be given the structure of an $r$-deformed $G \times  B$-equivariant Picard algebroid on $Y$.} 
\end{itemize}
\end{lemma}

\begin{proof}

For the first claim, we start by proving that $\widetilde{\mathcal{L}}=(f_* \mathcal{L})^G$ is a $B$-equivariant Lie algebroid. First, we get a $B$-action on $\widetilde{\mathcal{L}}$ from Lemma \ref{d1Omoduledescentpreserveequivariance}, so that $\widetilde{\mathcal{L}}$ becomes a $B$-equivariant $\mathcal{O}_X$-module. Further, since $\mathcal{L}$ is $B$-equivariant, axiom $i)$ of definition \ref{d1defofeqalg} is also satisfied.

Next, since the actions of $G$ and $B$ on $Y$ commute the anchor map $\sigma: \widetilde{\mathcal{T}}_X \to \mathcal{T}_X$ is $B$-equivariant. Composing with the $B$-equivariant map $\tilde{\rho}:\widetilde{\mathcal{L}} \to \widetilde{\mathcal{T}}_X$ we obtain a $B$-equivariant map $\sigma \circ \tilde{\rho}:\widetilde{\mathcal{L}} \to \mathcal{T}_X$, so axiom $ii)$ of \ref{d1defofeqalg} is satisfied.

We let $i_{\mathfrak{g}}$ and $i_{\mathfrak{b}}$ denote the restriction of $i_{\mathfrak{g} \times \mathfrak{b}}$ to $r\mathfrak{g} \cong r\mathfrak{g} \times 0 \subset r(\mathfrak{g} \times \mathfrak{b})$ and $r\mathfrak{b} \cong 0 \times r\mathfrak{b} \subset r(\mathfrak{g} \times \mathfrak{b})$, respectively. Let $\beta:\mathfrak{b}: \to \mathcal{T}_Y$, $\gamma: \mathfrak{b} \to \mathcal{T}_X$ denote the infinitesimal action $B$ on $Y$ and $X$; since $\mathcal{L}$ is in particular $B$-equivariant, we have $\beta=\rho \circ i_{\mathfrak{b}}$. By descending we obtain maps $\tilde{\beta}:r\mathfrak{b} \to \widetilde{\mathcal{T}}_X$ and $\tilde{i_{\mathfrak{b}}}:r\mathfrak{b} \to \widetilde{\mathcal{L}}$ such that $\tilde{\beta}= \tilde{\rho} \circ \tilde{i_{\mathfrak{b}}}$. Since the actions of $G$ and $B$ commute, we also get $\gamma=\sigma \circ \tilde{\beta}$ so $ \gamma= \sigma \circ \tilde{ \beta}$. Therefore combining this two we get:

$$ \gamma= (\sigma \circ \tilde{\rho}) \circ \tilde{i_{\mathfrak{b}}}.$$

Therefore we get, since $\mathcal{L}$ is $B$-equivariant, that the Lie algebroid $(\widetilde{\mathcal{L}},\sigma \circ \tilde{\rho}, \tilde{i_{\mathfrak{b}}})$ is also $B$-equivariant. By the proof of Lemma \ref{d1descenteqpicalg}, to show that $f_{\#} \mathcal{L}^G$ is $B$-equivariant it suffices to prove that $i_{\mathfrak{g}}(r\widetilde{\mathfrak{g}_X})$ is $B$-equivariant as an $\mathcal{O}_X$-module. Further, by using Lemma \ref{d1Omoduledescentpreserveequivariance} this is equivalent to proving $i_{\mathfrak{g}}(\mathcal{O}_Y \otimes r\mathfrak{g})$ is a $B$-equivariant $\mathcal{O}_Y$-module. This is true since $i_{\mathfrak{g} \times \mathfrak{b}}$ is in particular $B$-equivariant, so $i_{\mathfrak{g}}$ sends a $B$-equivariant module to a $B$-equivariant module.

Now, we prove the second claim. We may endow $\mathcal{K}$ with a trivial $G$-action $g.k=k$ for all $k \in \mathcal{K}$ and with the zero map $j_{\mathfrak{g}} \to \mathcal{K}$ so that $(\mathcal{K},\rho_X, j_{\mathfrak{g}} \times j_{\mathfrak{b}})$ becomes $G \times B$-equivariant. The claim follows from Lemma \ref{d1algrebroideqpullback}. 
\end{proof}

\begin{corollary}
The functors $f_{\#}(-)^G$ and $f^{\#}(-)$ induce  quasi-inverse equivalences of categories between the category of $r$-deformed $G \times B$-equivariant Picard algebroids on $Y$ and the category of $r$-deformed $B$-equivariant Picard algebroids on $X$.

\end{corollary}

\begin{proof}
This follows from Corollary \ref{d1noneqdescentpicalgcorr} and Lemma \ref{d1Bequivariantdescentpicalg}.
\end{proof}

\textbf{Descent of twisted differential operators}

We keep the notations from the start of the section. Further, we let $\mathcal{D}$ be a sheaf of $r$-deformed $G$-homogeneous twisted differential operators on $Y$.

\begin{definition}

We define the descent of $\mathcal{D}$ under the torsor $f:Y \to X$ to be the sheaf 

 $$f_{\#} \mathcal{D}^G:= \mathscr{T}( f_{\#} (\Lie \mathcal{D})^G).$$

\end{definition}

\begin{lemma}
\label{d1inversedescenthtdo}
Let the notations be as above. Then:
\begin{enumerate}[label=\roman*)]
\item{$f_{\#} \mathcal{D}^G$ is an $r$-deformed tdo on $X$.}
\item{$f^{\#}(f_{\#} \mathcal{D}^G) \cong \mathcal{D}$.}

\end{enumerate}

\end{lemma}

\begin{proof}

The first claim follows from Lemma \ref{d1descenteqpicalg} and Corollary \ref{d1Picardalgtdocorresp}. The second claim follows from Lemma \ref{d1descentinverseseqdesliealg} and Corollary \ref{d1htdoPicalgcorresp}.
\end{proof}

\begin{lemma}
\label{d1eqdescenthtdorightinverse}
Let $\mathcal{A}$ be an $r$-deformed tdo on $X$. Then $(f_{\#} f^{\#} \mathcal{A})^G \cong \mathcal{A}$.
\end{lemma}

\begin{proof}

We view $\mathcal{A}$ as an $r$-deformed $G$-equivariant htdo on $X$ with the trivial $G$-action, so that $\Lie(\mathcal{A})$ is an $r$-deformed $G$-equivariant Picard algebroid with the trivial $G$-action. The claim follows from Lemma \ref{d1eqdescentPicardrightinverse} and Corollary \ref{d1Picardalgtdocorresp}.
\end{proof}

Similarly to the deformed Picard algebroids case we obtain:

\begin{corollary}
\label{d1descentGhtdocorresp}
The functors $f_{\#}(-)^G$ and $f^{\#}(-)$ induce quasi-inverse equivalences between the category of $r$-deformed $G$-equivariant htdo's on $Y$ and the category of $r$-deformed tdo's on $X$.
\end{corollary}

\begin{proof}
This follows from the Lemmas \ref{d1inversedescenthtdo} and \ref{d1eqdescenthtdorightinverse}.
\end{proof}

\begin{corollary}
\label{d1eqdescentliealg}
Assume that $f$ is also $B$-equivariant. The functors $f_{\#}(-)^G$ and $f^{\#}(-)$ induce quasi-inverse equivalences between the category $r$-deformed $G \times B$-equivariant htdo's on $Y$ and the category $r$-deformed $B$-equivariant htdo's on $X$.
\end{corollary}

\begin{proof}
This follows from the previous corollary, along with Lemma \ref{d1Bequivariantdescentpicalg} and Corollary \ref{d1htdoPicalgcorresp}.
\end{proof}

\section{Equivariant descent for equivariant Picard algebroids and htdo's modules}
\label{d1sectiontdomoddescent}

We keep the notations from the previous section: recall that $G$ is a smooth affine algebraic group of finite type, $X$ and $Y$ are $R$-varieties, and $f:Y \to X$ is a locally trivial $G$-torsor. Further, we assume that $(\mathcal{L},\rho,i_{\mathfrak{g}}$ is an $r$-deformed $G$-equivariant Picard algebroid on $Y$.

\begin{lemma}
Let $\mathcal{M}$ be a $G$-equivariant Picard $\mathcal{L}$-module. Then $(f_* \mathcal{M})^G$ is a Picard $f_{\#} \mathcal{L}^G$-module. We call $(f_* \mathcal{M})^G$ the \emph{descent} of $\mathcal{M}$. 

\end{lemma}

\begin{proof}
Let $\widetilde{\mathcal{L}}=(f_* \mathcal{L})^G$. Then $(f_* \mathcal{M})^G$ is a Picard $\widetilde{\mathcal{L}}$-module, so it remains to prove that the action of $i_{\mathfrak{g}}(r\widetilde{\mathfrak{g}_X})$ kills $(f_* \mathcal{M})^G$. 

Since the $G$ action on $(f_* \mathcal{M})^G$ is constant, by differentiating this action we obtain a trivial $\mathfrak{g}$ action, so a trivial $r \mathfrak{g}$ action. But by our assumption on $\mathcal{M}$ this coincides with the $r\mathfrak{g}$ action induced from $ i_{\mathfrak{g}}: r\mathfrak{g} \to \mathcal{L}$; the conclusion follows.
\end{proof}

\begin{proposition}
\label{d1descentrepalgeqcat}

Let $\mathcal{L}$ be an $r$-deformed $G$-equivariant Picard algebroid on $Y$. The functors

\begin{equation}
\begin{split}
&f_* (-)^G: \Mod(\mathcal{L},G) \to \Mod(f_{\#} \mathcal{L}^G),\\
&f^{\#}(-): \Mod(f_{\#} \mathcal{L}^G) \to \Mod(\mathcal{L},G),
\end{split} 
\end{equation}

are quasi-inverse equivalences of categories.

\end{proposition}

\begin{proof}

This follows from  the Lemma above, Lemma \ref{d1descentinverseseqdesliealg} and Proposition \ref{d1Omodulesequivariantdescent}.
\end{proof}

\begin{corollary}
\label{d1noncohdescenttorsorequivalence}

Let $\mathcal{D}$ be an $r$-deformed $G$-htdo on $Y$. The functors

\begin{equation}
\begin{split}
&f_* (-)^G: \Mod(\mathcal{D},G) \to \Mod(f_{\#} \mathcal{D}^G),\\
&f^{\#}(-): \Mod(f_{\#} \mathcal{D}^G) \to \Mod(\mathcal{D},G),
\end{split} 
\end{equation}

are quasi-inverse equivalences of categories.
\end{corollary}

\begin{proof}

This follows from the Proposition above and Corollaries \ref{d1repequivalencenonequivliealgtdo} and \ref{d1repequivalentequivliealgtdo}.
\end{proof}

Let $B$ is another smooth affine algebraic group acting on $Y$ such that the actions of $G$ and $B$ commute and $f:Y \to X$ is $B$-equivariant. Let $\mathcal{D}$ be an $r$-deformed $G \times B$-htdo. Similar to the $\mathcal{O}$-module case (Corollary \ref{d1BequivarianteqdescentOmodules}), we obtain using Corollary \ref{d1noncohdescenttorsorequivalence}:

\begin{corollary}
\label{d1eqnoncohdescenttorsorequivalence}
Let $\mathcal{D}$ be an $r$-deformed $G \times B$-htdo on $Y$ and assume $f$ is $B$-equivariant. The functors

\begin{equation}
\begin{split}
&f_* (-)^G: \Mod(\mathcal{D},G \times B) \to \Mod(f_{\#} \mathcal{D}^G,B),\\
&f^{\#}(-): \Mod(f_{\#} \mathcal{D}^G,B) \to \Mod(\mathcal{D},G \times B),
\end{split} 
\end{equation}
are quasi-inverse equivalences of categories.
\end{corollary}

To prove the main theorem of the paper, we need two easy results. Recall that for an $R$-variety $Z$ we denote $\mathcal{D}_Z$ the sheaf of crystalline differential operators on $Z$. For $r \in R$ regular element, we call $\mathcal{D}_{Z,r}=\mathscr{T}(\mathcal{O}_Z \oplus r \mathcal{T}_Z)$ the sheaf of $r$-deformed differential operators and $\mathcal{D}(Z)_r=\Gamma(Z,\mathcal{D}_{Z,r})$. It is clear by construction that $\mathcal{D}_{Z,r}$ is an $r$-deformed tdo. We also make the following assumption:

\begin{assumption}
The base ring $R$ is Noetherian.
\end{assumption}

\begin{lemma}
Let $\mathcal{M}$ be a $\mathcal{D}_{X,r}$-module. Then $f^* \mathcal{M}$ is a coherent $\mathcal{D}_{Y,r}$-module if and only $\mathcal{M}$ is a coherent $\mathcal{D}_{X,r}$-module. 
\end{lemma}

\begin{proof}

Since coherence is a local problem, we may assume that $X$ is affine, $Y=G \times X$, $f$ is the projection onto the second factor and $G$ acts on $Y$ via left multiplication on the first factor.  Let $M$ be a $\mathcal{D}(X)_r$ module, so $f^*M \cong \mathcal{O}(G) \otimes M$.

We may view $D(Y)_r$ as an $r$-deformation of the ring $D(Y)$. Since $D(Y) \cong D(G) \otimes D(X)$ and deformations commutes with tensor product, we obtain that the ring $\mathcal{D}(Y)_r \cong \mathcal{D}(G)_r \otimes \mathcal{D}(X)_r$ acts naturally on $f^* M$. Further, since our base ring  $R$ is Noetherian we have by Lemma \ref{d1crystallinedifoplocnoether} that $D(Y)_r$ and $D(X)_r$ are Noetherian rings, so coherence is equivalent to local finite generation. Thus it is enough to prove the following claim:

\textbf{Claim}. $M$ is a finitely generated $\mathcal{D}(X)_r$ module if and only if $\mathcal{O}(G) \otimes M$ is a finitely generated $\mathcal{D}(G)_r \otimes \mathcal{D}(X)_r$ module. 

The direct implication is trivial. For the reverse implication, we will prove the contrapositive: if $M$ is not Noetherian, neither is $\mathcal{O}(G) \otimes M$. Let $M_1 \subsetneq M_2 \subsetneq M_3 \ldots \subsetneq $ be an infinite ascending of $M$-submodules. Since $G$ is a smooth group, $\mathcal{O}(G)$ is a faithfully flat $R$-module, thus tensoring with $\mathcal{O}(G)$ preserves injections. In particular, we obtain an infinite chain

$$ \mathcal{O}(G) \otimes M_1 \subsetneq \mathcal{O}(G) \otimes  M_2 \subsetneq \mathcal{O}(G) \otimes M_3 \ldots \subsetneq $$

of submodules of $\mathcal{O}(G) \otimes M$. Therefore, the claim is proven.
\end{proof}

\begin{corollary}
\label{d1tdopullbackpreservecoherence}
Let $\mathcal{D}$ be an $r$-deformed $G$-htdo on $Y$ and let $\mathcal{M}$ be a $(f_{\#} \mathcal{D})^G$-module. Then $f^{\#} \mathcal{M}$ is a coherent $\mathcal{D}$-module if and only if $\mathcal{M}$ is coherent.

\end{corollary}

\begin{proof}

The claim is local, thus we may assume that $\mathcal{D}=\mathcal{D}_{Y,r}$ so that $(f_{\#} \mathcal{D})^G= \mathcal{D}_{X,r}$. The claim follows by the Lemma above and Lemma \ref{d1inversedescenthtdo}.
\end{proof}

We may now prove the main theorem.

\begin{theorem}
\label{d1maintheorem}

Assume the base ring $R$ is Noetherian. Let $G$ be a smooth affine algebraic group of finite type. Let $X,Y$ be $R$-varieties and let $f:Y \to X$ be a locally trivial $G$-torsor. Further, let $\mathcal{D}$ be a sheaf of $r$-deformed $G$-homogeneous twisted differential operators on $Y$. The functors:

\begin{equation}
\begin{split}
&f_* (-)^G: \Coh(\mathcal{D},G) \to \Coh(f_{\#} \mathcal{D}^G),\\
&f^{\#}(-): \Coh(f_{\#} \mathcal{D}^G) \to \Coh(\mathcal{D},G),
\end{split} 
\end{equation}

are quasi-inverse equivalences of categories between coherent $G$-equivariant $\mathcal{D}$-modules and coherent $(f_{\#} \mathcal{D})^G$-modules.

\end{theorem}

\begin{proof}
This follows from Corollaries \ref{d1noncohdescenttorsorequivalence} and \ref{d1tdopullbackpreservecoherence}.
\end{proof}
As a corollary, we obtain using Corollary \ref{d1eqnoncohdescenttorsorequivalence}:
\begin{corollary}
\label{d1eqdescentrep}
Let the assumptions be as above. Further, assume that $B$ is a smooth affine algebraic group of finite type such that $f:Y \to X$ is $B$-equivariant. Further assume that $\mathcal{D}$ is an $r$-deformed $G \times B$-htdo on $Y$. The functors:

\begin{equation}
\begin{split}
&f_* (-)^G: \Coh(\mathcal{D},G \times B) \to \Coh(f_{\#} \mathcal{D}^G,B),\\
&f^{\#}(-): \Coh(f_{\#} \mathcal{D}^G,B) \to \Coh(\mathcal{D},G \times B),
\end{split} 
\end{equation}

are quasi-inverse equivalences of categories. 
\end{corollary}

\bibliography{duflo1}
\bibliographystyle{plain}

\end{document}